\documentclass[a4paper, 11pt, reqno]{amsart}
\usepackage{amsmath,amssymb,amsthm,amsfonts,mathrsfs,color,graphicx}
\usepackage{setspace}
\usepackage[margin=3cm]{geometry}

\usepackage[hidelinks]{hyperref}
\usepackage{cite}
\usepackage[normalem]{ulem}
\usepackage{enumerate,latexsym}

\DeclareMathOperator{\rr}{\mathbb R}

\newcommand{\ba}{\mathbb{B}^{n}(r)}

\newcommand{\tr}{\mathrm{Tr}}

\newcommand{\R}{\mathbb{R}}

\newcommand{\si}{\Sigma}

\newcommand{\p}{\partial}
\newcommand{\pa}{\partial}

\newtheorem{theorem}{Theorem}

\newtheorem*{thm1}{Theorem A}
\newtheorem*{thm2}{Theorem B}
\newtheorem*{thm3}{Theorem C}
\newtheorem{lemma}{Lemma}
\newtheorem{proposition}{Proposition}
\newtheorem{remark}{Remark}

\theoremstyle{definition}

\numberwithin{equation}{section}

\title{Eigenvalue problems and free boundary minimal surfaces in spherical caps}
\author{Vanderson Lima and Ana Menezes}
\address{Instituto de Matem\'atica e Estat\'istica\\ Universidade Federal do Rio Grande do Sul\\ Brazil}
\email{vanderson.lima@ufrgs.br}

\address{Department of Mathematics\\ Princeton University\\ USA}
\email{amenezes@math.princeton.edu}

\begin{document}

\maketitle

\begin{abstract}
Given a compact surface with boundary, we introduce a family of functionals on the space of its Riemannian metrics, defined via eigenvalues of a Steklov-type problem. We prove that each such functional is uniformly bounded from above, and we characterize maximizing metrics as induced by free boundary minimal immersions in some geodesic ball of a round sphere. Also, we determine that the maximizer in the case of a disk is a spherical cap of dimension two, and we prove rotational symmetry of free boundary minimal annuli in geodesic balls of round spheres which are immersed by first eigenfunctions.
\end{abstract}

\section{Introduction}

Given a closed surface $M$ and a Riemannian metric $g$ on $M$, let $\lambda_1(g)$ denote the first nontrivial eigenvalue of $-\Delta_g = -\mathrm{div}_{g}\nabla^{g}$. Define
\begin{equation}\label{eq:Lap.funct}
\lambda_{1}^{\ast}(M) = \sup_{g} \lambda_1(g)|M|_g,
\end{equation}
where $|M|_g$ denotes the area of $(M,g)$. By the work of Hersch \cite{He}, Yang-Yau \cite{YY} and Karpukhin \cite{K1}, we know that $\lambda_{1}^{\ast}(M) < \infty$. A natural question is whether there is a metric realizing $\lambda_{1}^{\ast}(M)$. 

Hersch \cite{He} showed that in the case of a sphere the only maximizers are round metrics. After that, Li-Yau \cite{LY} characterized the round metrics as the unique maximizers when $M = \mathbb{RP}^2$. Nadirashvili \cite{N} proved that metrics maximizing $\lambda_1(g)|M|_g$ are induced by minimal immersions $M \to \mathbb{S}^n$, for some $n$, and this led to the proof that the flat equilateral metric is the unique maximizer in the $2$-torus \cite{N,Gi,CKM}. The maximizers were also found for the Klein bottle \cite{EGJ,JNP,CKM}, and the orientable closed surface of genus $2$ \cite{JLNP,NS}. 

Concerning the existence of maximizers in general, Petrides \cite{Pe1} showed that, if instead of \eqref{eq:Lap.funct} we consider the quantity $\sup_{g \in \mathfrak{C}} \lambda_1(g)|M|_g$, for a fixed conformal class $\mathfrak{C}$, then a maximizer exists for every closed surface $M$ and it is induced by a harmonic map $(M,\mathfrak{C}) \to \mathbb{S}^{n}$. Denoting by $M_{\gamma}$ the orientable closed surface of genus $\gamma$, Petrides \cite{Pe1} also proved that if 
$\lambda_{1}^{\ast}(M_{\gamma+1}) > \lambda_{1}^{\ast}(M_{\gamma})$, then there exists a maximizer for $\lambda_{1}^{\ast}(M_{\gamma+1})$. In \cite{KS1,KS2} Karpukhin-Stern developed a completely new approach to the existence results, and in a recent paper \cite{KKMS}, Karpukhin-Kusner-McGrath-Stern proved that for every $\gamma \geq 0$, either $\lambda_{1}^{\ast}(M_{\gamma})$ or $\lambda_{1}^{\ast}(M_{\gamma+1})$ admits a $\lambda_{1}^{\ast}$-maximizing metric. 

Let $(\si,g)$ be a compact Riemannian surface with boundary. Fraser-Schoen \cite{F.S1,F.S2} considered the quantity
\begin{equation}\label{intro:Stek.funct}
\sigma_{1}^{\ast}(\si) = \sup_{g}\sigma_{1}(g)|\pa\si|_{g},
\end{equation}
where $|\pa\si|_{g}$ is the length of $(\pa\si, g)$ and $\sigma_{1}(g)$ denotes the first nontrivial Steklov eigenvalue, and they proved that $\sigma_{1}^{\ast}(\si) < \infty$ if $\si$ is orientable (see the work of Medvedev \cite{Me} for the non-orientable case). In \cite{F.S2}, they also showed that any metric achieving the supremum is induced by a free boundary conformal minimal immersion of $\si$ into an Euclidean unit ball $\mathrm{B}^{n}$. 

The maximizers of \eqref{intro:Stek.funct} are known in the case of the disk by the work of Weinstock \cite{We}, and in the cases of the annulus and the M\"obius band by Fraser-Schoen \cite{F.S2}. Moreover, combining \cite[Proposition 8.1]{F.S2} with the results in \cite{KS2}, we know that in the case of orientable surfaces of genus zero, the maximizers are realized by embedded free boundary minimal surfaces in $\mathrm{B}^3$ which converge to $\mathbb{S}^2$ as the number of boundary components goes to infinity. In the case of a general surface with boundary, Petrides \cite{Pe2} obtained existence results analogous to his previous results on the Laplacian problem mentioned before, and Karpukhin-Kusner-McGrath-Stern \cite{KKMS}  proved that each compact oriented surface with boundary of genus at most one admits a $\sigma_{1}^{\ast}$-maximizing metric.

Inspired by the previous two problems, we introduce an eigenvalue optimization problem which is related to free boundary minimal surfaces in \emph{spherical caps}, i.e., geodesic balls of round spheres. In order to state our results, let us first introduce some terminology.

Let $\alpha$ be a real number. There exists a self-adjoint operator $\mathcal{D}_{\alpha}$ (see Section \ref{sec:Steklov}) defined on a dense subspace of $L^{2}(\pa\si,g)$ whose spectrum is discrete and satisfies
$$\sigma_{0}(g,\alpha) < \sigma_{1}(g,\alpha) \leq \ldots \sigma_{j}(g,\alpha) \leq \ldots \to +\infty.$$
These numbers are called the \emph{Steklov eigenvalues with frequency} $\alpha$ of $(\si,g)$. The case $\alpha = 0$ recovers the Steklov spectrum. Observe that $\sigma_{0}(g,0) = 0$, for any metric $g$.

Consider the round sphere 
\begin{align*}
\mathbb{S}^{n} &= \bigg\{x \in \R^{n+1}; \ \sum_{j=0}^{n} x_{j}^{2} = 1 \bigg\},
\end{align*}
and let $\ba = \left\{x \in \mathbb{S}^{n};\, x_{0} \geq \cos r\right\}$ be its geodesic ball of center $\big(1,0,\ldots,0\big)$ and radius $0 < r < \frac{\pi}{2}$.  

Consider an isometric immersion $\Phi: (\si,g) \to \ba$ such that $\Phi(\partial\si) \subset \partial\ba$. Let $\nu$ be the outward pointing unit $g$-normal vector of $\partial\si$. Then $\Phi$ is minimal and free boundary if, and only if, the coordinate functions $\phi_i = x_i\circ\Phi$ satisfy (see Proposition \ref{prop-eigen}):
\begin{align*}
\Delta_{g}\, \phi_i + 2\phi_i &= 0 \quad \textrm{in } \si,\ i=0,1,\ldots,n,\\
\frac{\partial \phi_0}{\partial\nu} + (\tan r)\phi_0 &= 0\quad \textrm{on } \partial\si,\\
\frac{\partial \phi_i}{\partial\nu} - (\cot r)\phi_i &= 0 \quad \textrm{on } \partial\si,\ i=1,\ldots,n.
\end{align*}
Also, $2$ is not an eigenvalue of $-\Delta_g$ with Dirichlet boundary condition (see Proposition \ref{prop-dirichlet}). This implies that $\phi_{i}\big\vert_{\pa\si}$ are eigenfunctions of the Steklov problem with frequency $2$ (see Section \ref{sec:Steklov}), whose eigenvalues are $-\tan r$ (for $i=0$) and $\cot r$ (for $i = 1,\ldots,n$).

We denote by $\mathcal{M}(\si)$ the space of smooth Riemannian metrics $g$ on $\si$ such that the first eigenvalue of $-\Delta_g$ with Dirichlet boundary condition is greater than $2$ (see Section \ref{sec:Steklov} for a justification on this restriction on the metrics). Then, for any $g \in \mathcal{M}(\si)$, we define
\begin{align*}
\Theta_{r}(\si,g) &= \big[\sigma_{0}(g,2)\cos^{2} r + \sigma_{1}(g,2)\sin^{2} r\big]|\partial\si|_{g} +  2|\si|_{g}.
\end{align*}

\begin{thm1}
Let $\si$ be a compact orientable surface of genus $\gamma$ and $\ell$ boundary components. Then, for any $g \in \mathcal{M}(\si)$, we have
\begin{equation*}
\Theta_{r}(\si,g) \leq 4\pi(1 -\cos r)(\gamma + \ell).
\end{equation*}
Moreover, if $\si$ is a disk, the equality holds if, and only if, $(\si,g)$ is isometric to $\mathbb{B}^{2}_{r}$.
\end{thm1}

The uniqueness in the case of the disk is analogous to the result of Weinstock \cite{We}.

A consequence of the previous theorem is that if $\si$ is orientable, then the following quantities are finite:
\begin{align*}
\Theta_{r}(\si,\mathfrak{C}) &= \sup_{g \in \mathcal{M}(\si)\cap\mathfrak{C}} \Theta_{r}(\si,g),\\
\Theta_{r}^{\ast}(\si) &= \sup_{\mathfrak{C}} \Theta_{r}(\si,\mathfrak{C}) = \sup_{g \in \mathcal{M}(\si)} \Theta_{r}(\si,g),
\end{align*}
where $\mathfrak{C}$ denotes an arbitrary conformal class of $\si$.

It would be interesting to know whether maximizers exist for any topological type of $\Sigma$. In Section \ref{sec.max}, we prove that a maximizing sequence of metrics cannot converge smoothly to the boundary of $\mathcal{M}(\Sigma)$. In the next theorem, we show that if such maximizers exist, they are obtained as the induced metric of a free boundary minimal surface in a spherical cap.

\begin{thm2}\label{thm-intro2}
Let $\si$ be a compact surface with boundary.
\begin{itemize}
\item[(i)] If $g\in \mathcal{M}(\si)$ satisfies $\Theta_{r}(\si,g) = \Theta_{r}^{\ast}(\si)$, then there exist a $\sigma_{0}(g,2)$-eigenfunction $v_0$ and independent $\sigma_{1}(g,2)$-eigenfunctions $v_1,\ldots,v_n$, which induce a free boundary minimal isometric immersion $v = (v_0,v_1,\ldots,v_n):(\si,g) \to \ba$.
\item[(ii)] If $g\in \mathcal{M}(\si)\cap \, \mathfrak{C}$ satisfies $\Theta_{r}(\si,g) = \Theta_{r}(\si,\mathfrak{C})$, then there are a $\sigma_{0}(g,2)$-eigenfunction $v_0$ and independent $\sigma_{1}(g,2)$-eigenfunctions $v_1,\ldots,v_n$, which induce a free boundary harmonic map $v = (v_0,v_1,\ldots,v_n):(\si,g) \to \ba$.
\end{itemize}
\end{thm2} 

 In \cite{dCD}, do Carmo and Dajczer studied the concept of rotational hypersurfaces in $\mathbb{S}^n$ and characterized in the case $n=3$ the ones having constant mean curvature. The minimal surface case was done by Otsuki \cite{Ot}. These surfaces were considered by Li and Xiong \cite{LX} in the context of free boundary minimal surfaces in spherical caps. We will use here the setting of \cite{dCD}.

 In our next result, we characterize the immersions from Theorem B in the case of an annulus. 

\begin{thm3}\label{thm-intro3}
Let $\si$ be an annulus and consider a free boundary minimal immersion $\Phi = (\phi_0,\ldots,\phi_n):(\Sigma,g) \to \mathbb{B}^n(r)$. Suppose $\phi_i$ is a $\sigma_{1}(g,2)$-eigenfunction, for $i=1,\ldots,n$. Then $n=3$ and $\Phi(\Sigma)$ is one of the rotational examples given in \cite{dCD, Ot}.
\end{thm3}

Observe that we do not need to have as hypothesis that $\phi_0$ is a $\sigma_0(g,2)$-eigenfunction since this is always true for any free boundary minimal immersion into $\mathbb{B}^n(r)$ (see Remark \ref{rem-phi0}).

Theorem C is analogous to the uniqueness result of the critical catenoid proved by Fraser-Schoen in \cite{F.S3}, as well as to results of Montiel-Ros \cite{MoRo} and El Soufi-Ilias \cite{EI} which characterize the Clifford torus and the flat equilateral torus. 

Moreover, by Theorems B and C, if there is a smooth metric achieving $\Theta_{r}^{\ast}(\si)$ in the case of an annulus, then the metric is induced by the immersion of a rotational free boundary minimal annulus.

We should mention that after the works \cite{F.S1,F.S3}, the study of free boundary minimal surfaces in Riemannian manifolds has had a great progress, see for instance the survey \cite{Li} and the references therein. Let us now mention some other works connecting eigenvalue problems and minimal surfaces. We can define a functional analogous to \eqref{eq:Lap.funct} replacing $\lambda_1$ by a higher eigenvalue of $-\Delta_g$, which has been studied in \cite{EI2,Pe2,K2,KNPP1,KNPP2}. For the case of higher Steklov eigenvalues see \cite{F.S2,GP,Pe3}. Very recently, Petrides \cite{Pe4,Pe5} considered functionals involving either Laplacian eigenvalues or Steklov eigenvalues, whose maximizers are induced respectively by closed minimal surfaces in ellipsoids and compact free boundary minimal surfaces in domains bounded by ellipsoids. Also, Longa  \cite{Lo} linked compact minimal surfaces in spheres to maximizers of eigenvalues of the Neumann-Laplacian.

This paper is organized as follows. In the second section we describe the Steklov eigenvalue problem with frequency and we prove an upper bound for the multiplicity of the eigenvalues $\sigma_j(g, \alpha)$.  In Section 3 we describe the connection between free boundary minimal surfaces in spherical caps and the Steklov eigenvalue problem with frequency 2. In Section 4 we show the existence of an upper bound for our functional $\Theta_r(\Sigma, g)$ and the uniqueness in the case of a disk. In Section 5 we prove Theorem B. In Section 6 we describe the family of rotational free boundary minimal annuli in $\mathbb{B}^n(r)$ and, in Section 7, we prove the characterization of this family stated in Theorem C.
\\

\noindent
{\bf Acknowledgments}: 
The first author thanks the University of Chicago and Princeton University for the hospitality, where part of the research in this paper was developed. The first author was supported by the Simons Investigator Grant of Professor Andr\'e Neves, and by CNPq-Brazil Grant 405468/2021-0.\\


\section{The Steklov eigenvalue problem with frequency}\label{sec:Steklov}

Let $(\si,g)$ be a compact Riemannian manifold with boundary, and let $\nu_{g}$ be the outward pointing unit conormal vector of $\partial\si$. Fix $\alpha \in \rr$ which is not on the spectrum of $-\Delta_g$ with Dirichlet boundary condition. Then, given $u \in C^{\infty}(\partial\si)$, there is a unique $\widehat{u} \in C^{\infty}(\si)$ such that $\widehat{u}\vert_{\pa\si} = u$ and $\Delta_{g} \widehat{u} + \alpha  \widehat{u}= 0$ in $\si$. The \emph{Steklov eigenvalue problem at frequency $\alpha$} consists of finding $u \neq 0$ and $\sigma \in \rr$ such that
\begin{align*}
\frac{\partial \widehat{u}}{\partial\nu} = &  \ \sigma u \ \textrm{on } \partial\si.
\end{align*}

The numbers $\sigma$ for which there is $u$ satisfying the equation above form a sequence 
$$\sigma_{0}(g,\alpha) < \sigma_{1}(g,\alpha) \leq \ldots \sigma_{j}(g,\alpha) \leq \ldots \to +\infty,$$
and correspond to the spectrum of an operator $\mathcal{D}_{\alpha}$, which is called the \emph{Dirichlet-to-Neumann operator at frequency $\alpha$} and it is defined as follows:
\begin{itemize}
\item If $u \in H^{1}(\si,g)$ and $f \in L^{2}(\si,g)$, we say $\Delta_{g}u = f$ in the weak sense if
$$-\int_{\si} g\big(\nabla^{g} u,\nabla^{g} v\big)\,dA_{g} = \int_{\si} fv\,dA_{g}, \quad \forall\, v \in H^{1}_{0}(\si,g).$$
\item Suppose $u \in H^{1}(\si,g)$ and $\Delta_{g}u = f$ in the weak sense. We say that $\displaystyle\frac{\partial u}{\partial\nu_{g}} \in L^{2}(\partial\si,g)$ if there exists $\varphi \in L^{2}(\pa\si,g)$ such that
$$\int_{\si} g\big(\nabla^{g} u,\nabla^{g} v\big)\,dA_{g} + \int_{\si} fv\,dA_{g} = \int_{\pa\si} \varphi v\,dL_{g}, \quad \forall\, v \in H^{1}(\si,g).$$
\end{itemize}
Then, we define
\begin{align*}
\mathrm{dom}(\mathcal{D}_{\alpha}) &= \biggl\{u \in L^{2}(\partial\si,g);\, \exists\, \widehat{u} \in H^{1}(\si,g)\ \textrm{such that}\ \tr\,\widehat{u} = u,\, \Delta_{g} \widehat{u} + \alpha \widehat{u} = 0\\ &\quad\quad \textrm{in the weak sense, and}\ \frac{\partial \widehat{u}}{\partial\nu_{g}} \in L^{2}(\partial\si,g)\biggr\},\\
\mathcal{D}_{\alpha} u &= \frac{\partial \widehat{u}}{\partial\nu_{g}},
\end{align*}
where $\tr: H^{1}(\si,g) \to L^{2}(\pa\si,g)$ is the trace operator.

The domain $\mathrm{dom}(\mathcal{D}_{\alpha})$ is a dense subspace of $L^{2}(\partial\si,g)$ and $\mathcal{D}_{\alpha}$ is a self-adjoint operator which is bounded below and has compact resolvent. In particular, the spectrum of $\mathcal{D}_{\alpha}$ is discrete and bounded below. For more details about the construction of $\mathcal{D}_{\alpha}$, see \cite{AM1,AM2} (observe that their Laplacian is the negative of ours).

The eigenvalues $\sigma_{0}(g,\alpha)$ and $\sigma_{1}(g,\alpha)$ have a variational characterization. Let $Q$ be the quadratic form on $\mathrm{dom}(\mathcal{D}_{\alpha})$ defined by
$$Q(u,u) = \int_{\si} |\nabla^{g} \widehat{u}|^{2}_{g}\,dA_{g} - \alpha\int_{\si} \widehat{u}^{\,2}\,dA_{g}.$$
Then,
\begin{align}\label{var-sigma0}
\sigma_{0}(g,\alpha) = \inf\left\{Q(v,v)/\vert\vert u\vert\vert_{L^2(\pa\si)};\,u \in \mathrm{dom}(\mathcal{D}_{\alpha})\setminus\{0\}\right\},
\end{align}
and this eigenvalue is simple. Denote by $\phi_0$ a first eigenfunction, which we can choose to be positive. Then, applying the min-max principle to $\mathcal{D}_{\alpha}$, it follows that
\begin{equation*}\label{var-sigma1}
\sigma_{1}(g,\alpha) = \inf\left\{Q(v,v)/\vert\vert u\vert\vert_{L^2(\pa\si)};\,u \in \mathrm{dom}(\mathcal{D}_{\alpha})\setminus\{0\}\ \textrm{and} \int_{\partial\si} u \phi_{0}\,dL_{g} = 0\right\}.
\end{equation*} 

Now, we will present an alternative characterization of $\sigma_0(g,\alpha)$ and $\sigma_1(g,\alpha)$ for a special class of metrics. Denote by $\lambda_{0}^{D}(g)$ the first eigenvalue of $-\Delta_g$ with Dirichlet boundary condition. If $\lambda_{0}^{D}(g) > \alpha\geq0$, then
\begin{equation}\label{var-sigma)_alt}
\begin{split}
\displaystyle\sigma_{0}(g,\alpha) &= \inf\left\{\frac{\displaystyle\int_{\si} |\nabla^{g} u|^{2}_{g}\,dA_{g} - \alpha\int_{\si}u^{2}\,dA_{g}}{\displaystyle\int_{\partial\si} u^2\,dL_{g}};\,u \in W^{1,2}(\si)\ \text{and}\ \tr\,u \neq 0\right\},\\
\sigma_{1}(g,\alpha) &= \inf \left\{\frac{\displaystyle\int_{\si} |\nabla^{g} u|^{2}_{g}\,dA_{g} - \alpha\int_{\si} u^{2}\,dA_{g}}{\displaystyle\int_{\partial\si} u^2\,dL_{g}};\, u \in W^{1,2}(\si), \tr\,u \neq 0\ \textrm{and} \int_{\partial\si} u \phi_{0}\,dL_{g} = 0\right\}.
\end{split}
\end{equation}

\begin{remark}
For $\alpha \geq 0$, the quotient $\displaystyle \int_{\si}\big(|\nabla^{g} u|^{2}_{g} - \alpha u^{2}\big)\,dA_{g}/\int_{\partial\si} u^2\,dL_{g}$ is bounded from below in the domain $\displaystyle\{u \in W^{1,2}(\si), \tr\,u \neq 0\}$ if, and only if, $\lambda_{0}^{D}(g) > \alpha$. Hence, the previous characterization only holds in the case $\lambda_{0}^{D}(g) > \alpha$.
\end{remark}

Observe that the hypothesis $\lambda_{0}^{D}(g) > \alpha$ is equivalent to the condition that the Dirichlet eigenvalue problem 
\begin{equation*}
\begin{split}
-\Delta_{g}u - \alpha u &= \lambda u,\ \textrm{in}\ \si,\\
u &= 0,\ \textrm{on}\ \partial\si,
\end{split}
\end{equation*}
 does not admit non-positive eigenvalues $\lambda$.

In order to find an upper bound for the multiplicity of $\sigma_j(g,\alpha)$, $j\geq 1$,  we need the following result concerning nodal domains.

\begin{theorem}[Hassannezhad-Sher, \cite{HS}]\label{thm-HS}
Let $(\si,g)$ be a compact Riemannian manifold with boundary such that $\lambda_{0}^{D}(g) > \alpha$. Then, for any $\sigma_{j}(g,\alpha)$-eigenfunction, its number of nodal domains $N_j$ in $\si$ satisfies
$$N_j \leq j + 1.$$
\end{theorem}

Using this result we prove the following.

\begin{theorem} \label{thm:multiplicity}
Let $(\si,g)$ be a compact Riemannian surface of genus $\gamma$ and
$\ell$ boundary components such that $\lambda_{0}^{D}(g) > \alpha$.
Then, the multiplicity of $\sigma_{j}(g,\alpha)$ is at most $4\gamma + 2j + 1$.
\end{theorem}

\begin{proof}
The proof follows the same strategy as in \cite[Theorem 2.3]{F.S3}. 

Consider a function $\phi \in C^{\infty}(\si)$ satisfying
\begin{equation}\label{eq:eing.aux}
\begin{split}
\Delta_{g}\phi + \alpha\phi &= 0\ \textrm{in}\ \si,\\
\frac{\partial\phi}{\partial\nu_{g}} - \sigma_{j}\phi &= 0\ \textrm{on}\ \partial\si.
\end{split}
\end{equation}

\noindent
{\bf Claim 1}: The set $\mathcal{S}_{\phi} = \big\{p \in \si; \, \phi(p) = 0, \nabla^{g}\phi(p) = 0\big\}$ is finite.\\

The claim for $\big(\si\setminus\partial\si\big)\cap \mathcal{S}_{\phi}$ is a classical result and it follows as in \cite{Ch}. 

Let $p \in \big(\partial\si\big)\cap \mathcal{S}_{\phi}$ and consider a chart $\Psi : \mathcal{U} \to U \subset \R^{2}_{+}$, where $U$ is the half-disk centered at the origin, such that $\Psi(p) = (0,0)$, the associated coordinates $x,y$ are isothermal, and $\Psi_{\ast}\nu_{g}$ is on the direction of $\partial_{y}$.
By abuse of notation, we will also denote $\phi\circ\Psi^{-1}$ by $\phi$. In these coordinates, we have $g = \mu\delta$, where $\delta = dx^2 + dy^2$ and $\mu$ is a positive function. It follows from \eqref{eq:eing.aux} that $v = e^{\sigma_{j}\mu y}\phi$ satisfies 
\begin{equation*}
\Delta_{\delta}v + \langle b,\nabla^{\delta}v\rangle + c v = 0\ \textrm{on}\ U,\, \textrm{and}\ \ \frac{\partial v}{\partial y}(x,0) = 0,
\end{equation*}
where $b$ and $c$ are smooth and bounded up to the boundary. The remaining of the proof follows exactly as in \cite{F.S3}, extending $v$ by reflection on the boundary and using unique continuation results.\\

\noindent
{\bf Claim 2}: Consider $p \in \si\setminus\partial\si$ such that $\phi(p) = 0$. Then, the order of vanishing of $\phi$ at $p$ is at most $2\gamma + j$.\\

Again the proof follows as in \cite{F.S3}, the only difference being that in \cite{F.S3} the authors make use of the Nodal Domain Theorem for the Steklov problem (the case $\alpha = 0$), and in our case we need to use Theorem \ref{thm-HS}.
\\

\noindent
{\bf Claim 3}: The multiplicity of $\sigma_{j}$ is at most $4\gamma + 2j + 1$.\\

Let $m_j$ be the multiplicity of $\sigma_j$ and consider a basis $\{\phi_{1},\ldots,\phi_{m_{j}}\}$ of the $j$-th eigenspace. Suppose that $m_j > 4\gamma + 2j + 1$. The strategy is to show this implies the existence of a $\sigma_j$-eigenfunction $\phi = a_{1}\phi_{1} + \ldots + a_{m_{j}}\phi_{m_{j}}$ and a point $p \in \si\setminus\partial\si$ such that the order of vanishing of $\phi$ at $p$ is greater than $2\gamma + j$, which contradicts Claim 2.

Let $p \in \si\setminus\partial\si$ and consider isothermal coordinates $(x,y)$ around $p$ such that $p$ is mapped to $(0,0)$ and $g = \mu\delta$. Since $\Delta_{g}\,\phi_{j} + \alpha\phi_{j} = 0$, for $0 \leq l \leq q$, we have
$$\frac{\pa^{q}\Delta_{\delta}\,\phi_{j}}{\pa^{l}x\, \pa^{q-l}y} = L\phi_{j},$$
where $L$ is a linear operator of order $q$.
Supposing that all the derivatives of $\phi_j$ of order $0\leq l \leq q+1$ vanish at $(0,0)$ we obtain
$$\frac{\pa^{q+2}\phi_{i}}{\pa^{l+2}x\, \pa^{q-l}y}(0,0) + \frac{\pa^{q+2}\phi_{i}}{\pa^{l}x\, \pa^{q-l+2}y}(0,0) = 0.$$
Now we can argue as in \cite{F.S3} to finish the proof.
\end{proof}

\section{Free boundary minimal immersions in spherical caps}

We start this section with two results which relate free boundary minimal surfaces in spherical caps to eigenvalue problems. In the following, we denote by $\{\partial_0,\partial_1,\ldots,\partial_n\}$ the canonical basis of $\R^{n+1}$. If $(\si,g)$ is a Riemannian manifold with boundary, and $\Phi: (\si,g) \to \ba$ is an isometric immersion with $\Phi(\pa\si) \subset \pa\ba$, we say $\Phi$ is \emph{free boundary} if $\si$ and $\pa\ba$ intersect orthogonally.
 
\begin{proposition}\label{prop-eigen}
Let $(\si^{k},g)$ be a compact Riemannian manifold and let $\nu$ be the outward pointing unit $g$-normal vector of $\partial\si$. Let $\Phi: (\si,g) \to \ba$ be an isometric immersion such that $\Phi(\partial\si) \subset \partial\ba$. Then $\Phi$ is minimal and free boundary if, and only if, the coordinate functions $\phi_i = x_i\circ\Phi$ satisfy 
\begin{align}
\Delta_{g}\, \phi_i + k\phi_i &= 0,\quad \textrm{in } \si,\ i=0,1,\ldots,n,\label{eq:FBMS1}\\
\frac{\partial \phi_0}{\partial\nu} + (\tan r)\phi_0 &= 0,\quad \textrm{on } \partial\si,\label{eq:FBMS2}\\
\frac{\partial \phi_i}{\partial\nu} - (\cot r)\phi_i &= 0,\quad \textrm{on } \partial\si,\ i=1,\ldots,n.\label{eq:FBMS3}
\end{align}
\end{proposition}

\begin{proof}
Let $\Phi: (\si^k,g) \to \ba$ be an isometric immersion such that $\phi_{0}\big\vert_{\partial\si} \equiv \cos r$. Fix $v \in \R^{n+1}$, and consider the function $\phi_{v}: \si \to \R$ defined by $\phi_{v}(x) = \langle \Phi(x),v\rangle$. It is well known that (see for instance \cite{Ta})
$$\Delta_{g}\,\phi_{v} = \langle \vec{H},v\rangle -k\phi_{v}.$$
Thus $\vec{H} = 0$ if, and only if, $\Delta_{g}\,\phi_{v} + k\phi_{v} = 0$, for all $v \in \R^{n+1}$.

On the other hand, 
$$\frac{\partial \phi_{v}}{\partial\nu} = \langle \nu,v\rangle.$$
The outward pointing unit normal to $\partial\ba$ is given by 
$$N_{\pa\ba} = \frac{1}{\sin r}\big(\cos r\,x - \partial_0\big),$$ 
and the immersion $\Phi$ is free boundary if, and only if, $\nu = N_{\pa\ba}$, i.e., if and only if
\begin{equation}\label{eq:der.normal}
\frac{\partial \phi_{v}}{\partial\nu} = \frac{1}{\sin r}\big(\cos r\langle \Phi,v\rangle - \langle \partial_0,v\rangle\big).
\end{equation}
Choosing $v = \partial_i$, $i=0,1,\ldots,n$, we obtain \eqref{eq:FBMS2} and \eqref{eq:FBMS3}.
\end{proof}

\begin{proposition}\label{prop-dirichlet}
Let $(\si^k,g)$ be a compact Riemannian manifold and $\Phi: (\si,g) \to \ba$ be an isometric minimal immersion with free boundary. Then, the Dirichlet problem
$$
\begin{cases}
-\Delta_g w -kw=\lambda w, & \mbox{in } \Sigma, \\
w=0, & \mbox{on } \partial\Sigma ,
\end{cases}
$$ 
has no solution for any $\lambda \leq 0$, i.e., $\lambda_0^D(g)>k$ . In particular, for any function $u \in C^{\infty}(\si)$ that satisfies $u=0$ on $\partial \si$, it holds that
$$
\int_{\si}(|\nabla^g u|^2-ku^2)dA\geq 0.
$$
\end{proposition}

\begin{proof} 
By Proposition \ref{prop-eigen}, we know $\phi_0=\langle\Phi, \partial_0\rangle$ satisfies $\Delta \phi_0 +k\phi_0=0$ with $\phi_0>0$ in $\Sigma.$
 
Suppose, by contradiction, there exists a nontrivial solution $w$ to the Dirichlet problem above with $\lambda\leq 0$. We can write $w=h\phi_0,$ for some function $h,$ where $h= 0$ on $\partial\si$. Hence,
\begin{align*}
0&=\displaystyle\int_\si (\vert \nabla w\vert^2-(k+\lambda)w^2)dA\\
&= \displaystyle\int_{\si}\big(\vert \nabla h\vert^2\phi_0^2+2h\phi_0\langle \nabla h, \nabla \phi_0\rangle +h^2\vert \nabla\phi_0\vert^2-(k+\lambda)h^2\phi_0^2\big)dA.
\end{align*}
Applying the divergence theorem and using that $\Delta \phi_0 +k\phi_0=0$ in $\si$ and $h=0$ on $\partial \si$, we get
$$
\displaystyle\int_\si h\phi_0\langle \nabla h, \nabla \phi_0\rangle dA = -\displaystyle\int_\si \left( h\phi_0\langle \nabla h, \nabla \phi_0\rangle + h^2\vert \nabla\phi_0\vert^2 + h^2\phi_0\Delta\phi_0\right)dA +\int_{\partial\si}h^2\phi_0\frac{\partial \phi_0}{\partial \nu} dL
$$

$$
\implies 2\displaystyle\int_\si h\phi_0\langle \nabla h, \nabla \phi_0\rangle dA =-\displaystyle\int_\si  h^2\vert \nabla\phi_0\vert^2 dA + k\displaystyle\int_\si  h^2\phi_0^2 dA.
$$
Then,
$$
0=\displaystyle\int_\si (\vert \nabla w\vert^2-(k+\lambda)w^2)dA=\displaystyle\int_\si \vert \nabla h\vert^2\phi_0^2dA-\lambda\displaystyle\int_\si  h^2\phi_0^2 dA.
$$
Since $\lambda\leq 0$ and $\phi_0>0,$ it implies that $h$ is necessarily a constant function. Thus, since $h=0$ on $\partial\si$, we conclude that $h\equiv 0$ in $\si$ and, therefore, $w\equiv 0$ in $\si$, a contradiction.
\end{proof}

\begin{remark}
Theorem \ref{thm:multiplicity} and Proposition \ref{prop-dirichlet} imply that for any free boundary minimal immersion $\Phi:(\si^2,g) \to \ba$ the multiplicity of $\sigma_j(g,2)$ is at most $4\gamma+2j+1$.\label{rem-mult}
\end{remark}

\begin{remark}
It follows from Propositions \ref{prop-eigen} and \ref{prop-dirichlet} that the coordinates functions of a free boundary minimal immersion $\Phi:(\si,g) \to \ba$ are Steklov eigenfunctions with frequency $2$. Also, since the eigenfunction $\phi_0$ associated to the eigenvalue $\sigma = -\tan r$ is positive, we have $\sigma_{0}(g,2) = -\tan r$.
\label{rem-phi0}
\end{remark}

Now we specialize to the codimension one case. Let $\Phi: \Sigma^{n-1} \to \mathbb B^{n}(r)$ be a free boundary minimal immersion and suppose $\si$ is orientable. The index form associated to the second variation of area is given by
\begin{equation*}\label{eq:index-form}
\begin{split}
S(\psi,\phi)&=\int_\Sigma\big[\langle\nabla^{g} \psi,\nabla^{g}\phi\rangle - (n-1 + \vert\mathrm{I\!I}_{\si}\vert^2)\psi\phi \big]dA - \cot(r)\int_{\partial\Sigma}\psi\phi\, dL,\\
&=-\int_{\Sigma}\big(\Delta_{g} \psi + (n-1)\,\psi + \vert\mathrm{I\!I}_{\si}\vert^2\psi\big)\phi\,dA + \int_{\partial\Sigma}\bigg(\frac{\pa\psi}{\pa\nu} -  \cot(r)\psi\bigg)\phi\, dL,
\end{split}
\end{equation*}
where $\mathrm{I\!I}_{\si}$ is the second fundamental of $\si$ in $\ba$. Observe that if $\si^{n-1}$ is not contained in a hyperplane containing the origin, then the space of functions $\mathcal C=\{\langle \Phi, v \rangle; v\in \mathbb R^{n+1}\}$ has dimension $n+1$.

\begin{lemma}
\label{lemma:perp}
Let $\Phi: \Sigma^{n-1} \to \mathbb B^{n}(r)$ be as above. Then, for any $\psi\in \mathcal C,$ we have
$$S(\psi,\psi) = -\int_{\si}\vert\mathrm{I\!I}_{\si}\vert^2\psi^2dA-\langle v, \partial_0 \rangle^2 \cot(r) \vert\partial \si\vert,$$
where $v$ is the vector that defines $\psi,$ i.e., $\psi=\langle \Phi, v \rangle.$
Hence, if $\si^{n-1}$ is not contained in a hyperplane containing the origin, then its Morse index is at least $n+1$.
\end{lemma}

\begin{proof}
First observe that since $\Delta_{g} \psi+(n-1)\psi=0,$ we just need to compute the boundary term. 
Using that $\nu=\displaystyle\frac{1}{\sin r}(\cos(r)\Phi-\partial_0),$ we have
\begin{align*}
\int_{\partial\si} \langle \Phi, v\rangle\frac{\pa}{\pa\nu} \big(\langle \Phi, v\rangle\big) dL -\cot(r)\int_{\partial \si} \langle \Phi, v \rangle^2 dL&= \int_{\partial\si} \langle \nu, v\rangle\langle \Phi, v\rangle dL -\cot(r)\int_{\partial \si} \langle \Phi, v \rangle^2 dL\\
&= -\frac{\langle v, \partial_0\rangle}{\sin r}\int_{\partial\si}\langle\Phi, v\rangle dL\\
&= -\frac{\langle v, \partial_0\rangle}{\sin r}\int_{\partial\si}\left[ \langle v, \partial_0\rangle\phi_0+ \sum_{i=1}^3\langle v,\partial_i\rangle\phi_i \right] dL\\
&= -\langle v, \partial_0\rangle^2\cot(r)\vert\partial \si\vert,
\end{align*}
where in the last equality we use that $\phi_0\big\vert_{\partial\si}=\cos(r)$ and the fact that $\displaystyle\sum_{i=1}^3\langle v,\partial_i\rangle\phi_i$ is a $\sigma_j$-eigenfunction for some $j \geq 1$ and, therefore, it is $L^2(\partial\si)$-orthogonal to $\phi = 1$ which is a
 $\sigma_0$-eigenfunction.
\end{proof}

\section{Eigenvalue bounds and a Weinstock-type result}

Let $\si$ be a compact surface with  boundary. We will denote by $\mathcal{M}(\si)$ the space of smooth Riemannian metrics $g$ on $\si$ such that $\lambda_{0}^{D}(g) > 2$. From now on, we will only be interested in metrics satisfying this property; hence, for simplicity, on the remainder of the paper we will denote $\sigma_{k}(g) = \sigma_{k}(g,2)$, $k \geq 0$.

Fix $r\in (0,\frac{\pi}{2}).$ We define the functional
\begin{align}\label{eq:funct_r}
\Theta_{r}(\si,g) &= \big[\sigma_{0}(g)\cos^{2} r + \sigma_{1}(g)\sin^{2} r\big]|\partial\si|_{g} +  2|\si|_{g}.
\end{align}

The goal of this section is to prove the existence of an upper bound for the functional $\Theta_r$ that is sharp in the case of a disk. For that, we will first prove that $\mathbb B^2(r)$ is immersed by eigenfunctions of $\sigma_0$ and $\sigma_1$. More precisely, we show:

\begin{theorem}
For any $0<r<\frac{\pi}{2}$, $\mathbb B^2(r)$ satisfies $\sigma_0=-\tan r$ and $\sigma_1=\cot r.$
\label{thm-disk}
\end{theorem}

\begin{proof}

Consider the parametrization $\Phi: [-r, r]\times \mathbb S^1\to \mathbb B^2(r)$ given by 
$$\Phi(t,\theta) = (\cos t, \sin t \cos\theta,\sin t \sin \theta).$$ By Proposition \ref{prop-eigen}, we already know that the coordinate functions $\phi_i$ are eigenfunctions.
Since $\phi_0(t,\theta)=\cos t$ is positive, then $\phi_0$ is necessarily a $\sigma_0$-eigenfunction with $\sigma_0=-\tan(r)$. Now let us show that $\sigma_1=\cot r.$

We have,
\begin{align*}
\Phi_{t} &= (-\sin t, \cos t \cos\theta,\cos t \sin \theta),\\
\Phi_{\theta} &
= (0,-\sin t \sin\theta,\sin t \cos\theta),\\
g_{tt} &= 1, g_{t\theta} = 0, g_{\theta\theta} = \sin^2 t.
\end{align*}
Hence, 
$$\Delta_g f = \pa^{2}_{tt}f + \frac{1}{\sin^2 t}\pa^{2}_{\theta\theta}f + (\cot t) \pa_{t}f.$$

Since the Laplace eigenfunctions of $\mathbb S^1$ form a basis for $L^2$, we can write any function $f:[-r, r]\times\mathbb S^1\to \mathbb{R}$ as
\begin{equation}
f(t,\theta)=a_0(t)\cdot 1+\displaystyle\sum_{k=1}^{\infty}[a_k(t)\cos(k\theta)+b_k(t)\sin(k\theta)].
\end{equation}

We claim that, if $f$ is a $\sigma_1$-eigenfunction then the functions $a_0(t)$, $a_k(t)\cos(k\theta)$, $b_k(t)\sin(k\theta)$ are $\sigma_1$-eigenfunctions as well. In fact, observe that 
$$
\Delta_g (a_k(t)\cos(k\theta))= \Delta_g (a_k(t))\cdot \cos(k\theta)-k^2\frac{a_k(t)}{\sin^2t}\cos(k\theta)=: \tilde{a}_k(t)\cos(k\theta),
$$
and, similarly, $\Delta_g (b_k(t)\sin(k\theta))=\tilde{b}_k(t)\sin(k\theta).$ Hence, $\Delta_g f+2f=0$ if, and only if, each of the functions $a_0(t)$, $a_k(t)\cos(k\theta)$, $b_k(t)\sin(k\theta)$ satisfies this same equation. Moreover, since $\frac{\partial}{\partial\nu}=\pm\partial_t$, they all have the same eigenvalue.

Also, by Theorem \ref{thm-HS}, we know the $\sigma_1$-eigenfunctions have at most $2$ nodal domains; hence, in particular, the functions $a_k$ and  $b_k$ are identically zero for $k\geq2.$

Notice that $a_0$ and $a_1$ satisfy:
\begin{align}
a_{0}^{\prime\prime}(t) + (\cot t)a_{0}^{\prime}(t) + 2a_{0}(t) &= 0,\label{eq:azero}\\
a_{1}^{\prime\prime}(t) + (\cot t)a_{1}^{\prime}(t) + \left[2 - \frac{1}{\sin^2 t}\right]a_{1}(t) &= 0.\label{eq:aone}
\end{align}

The function $\cos t$ is a solution of \eqref{eq:azero}. Thus, a second solution linearly independent to $\cos t$ is of the form $y = h(t)\cos t$, where $h$ satisfies
$$(\cos t)h^{\prime\prime}(t) + \left(-2\sin t + \frac{\cos^2 t}{\sin t}\right)h^{\prime}(t) = 0.$$
Hence,
$$y = 1 + \frac{1}{2}(\cos t)\log\left(\frac{1 - \cos t}{1 + \cos t}\right).$$
This function does not extend continuously to $t = 0$. Therefore, $a_0(t)$ is a constant multiple of $\cos t$; but since $\cos t$ is a $\sigma_0$-eigenfunction, we conclude $a_0 \equiv 0$ necessarily.

On the other hand, $\sin t$ is a solution of \eqref{eq:aone}. Thus, a second solution linearly independent to $\sin t$ is of the form $z = h(t)\sin t$, where $h$ satisfies
$$(\sin t)h^{\prime\prime}(t) + \left(3\cos t\right)h^{\prime}(t) = 0.$$
Hence,
$$z = \frac{\sin t}{8}\left[\sec^2 \left(\frac{t}{2}\right) + 4\log\left(\tan \left(\frac{t}{2}\right)\right) - \csc^2\left(\frac{t}{2}\right)\right].$$
This function also does not extend continuously to $t = 0$. Thus, $a_1(t)$ is a constant multiple of $\sin t$. By a similar analysis, we conclude $b_1(t) = c\sin t$ for some constant $c$. Therefore, $\phi_1=\sin t \cos\theta$ and $\phi_2=\sin t \sin \theta$ are $\sigma_1$-eigenfunctions and $\sigma_1=\cot r.$
\end{proof}

Next we prove the main result of this section that was stated in the introduction.

\begin{thm1}
Let $\si$ be a compact orientable surface of genus $\gamma$ and $\ell$ boundary components. Then, for any $g \in \mathcal{M}(\si)$, we have
\begin{equation}\label{eq:bound.func}
\Theta_{r}(\si,g) \leq 4\pi(1 -\cos r)(\gamma + \ell).
\end{equation}
Moreover, if $\si$ is a disk, then the equality in \eqref{eq:bound.func} holds if, and only if, $(\si,g)$ is isometric to $\mathbb{B}^{2}_{r}$.
\end{thm1}

\begin{proof}
Let $\si$ be a compact orientable surface of genus $\gamma$ and $\ell$ boundary components. By a result of Gabard \cite{Ga}, there exists a proper conformal branched cover $w:\Sigma\to \mathbb{D}^2$ of degree at most $\gamma+\ell$, where $\mathbb{D}^2\subset\mathbb{R}^2$ is the closed unit disk with center at the origin. 

Fix $r \in (0,\frac{\pi}{2})$. Recall that the area of $\mathbb{B}^{2}_{r}$ with respect to the round metric is $2\pi(1-\cos r)$. Since $\mathbb{D}^2$ is conformally equivalent to $\mathbb{B}^{2}_{r}$, there is also a proper conformal branched cover $u:(\Sigma,g) \to \mathbb{B}^{2}_{r}$ of the same degree as $w$.

Denote by $x_i$ and $\partial_i,\, i=0,1,2$,  the canonical coordinates of $\rr^3$ and the corresponding coordinate vector fields. 
Define $u_j = x_j\circ u$. In particular, along $\partial\si$ it holds
\begin{equation}\label{eq:Hersch.bdr}
u_0 = \cos r,\quad u_{1}^{2} + u_{2}^{2} = \sin^{2} r.
\end{equation}
Moreover, by using conformal diffeomorphisms of $\mathbb{B}^{2}_{r}$, we can also assume (see \cite{LY})
$$\int_{\partial\Sigma}\ u_j \phi_0\,dL=0, \quad j=1,2,$$ 
where $\phi_0$ is a positive eigenfunction associated to $\sigma_{0}$.

Since $u: \si \to \mathbb{B}^{2}_{r}$ is conformal, there is a positive function $f \in C^{\infty}(\si)$ such that, for all $X,Y \in \mathfrak{X}(\si)$, we have
$$g_{\mathbb{B}^{2}_{r}}\big(u_{*}X,u_{*}Y) = f^2 g(X,Y).$$  
So, by abuse of notation, 
$$\nabla^{g}u_j = f^2\nabla^{\mathbb{B}^{2}_{r}} x_j = f^2\big(\partial_j - x_{j} x\big),$$ 
which implies
\begin{equation}\label{eq:conformal.grad}
\sum_{j=0}^{2}\vert \nabla^{g}u_{j}\vert^{2}_{g} = 2f^2.
\end{equation}
Thus,
\begin{equation}\label{eq:Hersch.Dir}
\sum_{j=0}^2 \int_{\si} |\nabla^{g}u_j|^{2}_{g}\,dA_{g} = 2\int_{\si}f^2dA_{g}= 4\,\mathrm{deg}(u)\pi(1-\cos r) \leq  4\pi(1-\cos r)(\gamma + \ell).
\end{equation}

By the variational characterization of the eigenvalues, we have
\begin{align}\label{ineq1}
\sigma_{0}(g)\int_{\partial\si} u_{0}^{2}\,dL_{g} \leq \int_{\si} |\nabla^{g}u_{0}|^{2}_{g}\,dA_{g} - 2\int_{\si} u_{0}^{2}\,dA_{g}, 
\end{align}
and for $j = 1,2$,
\begin{align}\label{ineq2}
\sigma_{1}(g) \int_{\partial\si} u_{j}^{2}\,dL_{g} \leq \int_{\si} |\nabla^{g}u_{j}|^{2}_{g}\,dA_{g} - 2\int_{\si} u_{j}^{2}\,dA_{g}.
\end{align}
Summing these inequalities and using \eqref{eq:Hersch.bdr} and \eqref{eq:Hersch.Dir}, we obtain
\begin{equation*}\label{eq:estimate}
\left(\sigma_{0}(g)\cos^{2} r + \sigma_{1}(g)\sin^{2} r\right)|\partial\si|_{g} + 2|\si|_{g} \leq 4\pi(1-\cos r)(\gamma + \ell).
\end{equation*}

Now we will show that in the case the surface is a disk the equality holds if, and only if, the surface is isometric to the spherical cap.

First let us suppose $\Sigma=\mathbb{B}^{2}_{r}$. Then, by Theorem \ref{thm-disk}, we have 
\begin{align*}
\sigma_0(g) =  -\tan r,\quad  \sigma_1(g) =  \cot r.
\end{align*}

Hence, $\left(\sigma_{0}(g)\cos^{2} r + \sigma_{1}(g)\sin^{2} r\right)|\partial\si|_{g} + 2|\si|_{g}  = 4\pi(1-\cos r)$.

Now suppose the equality occurs in $(\ref{eq:bound.func})$ and $\gamma=0$ and $\ell=1$. Then, in particular, all inequalities above are equalities. From (\ref{eq:Hersch.Dir}), we conclude that $\mathrm{deg}(u)=1$.  Equalities in (\ref{ineq1}) and (\ref{ineq2}) imply that $u_0$ is an eigenfunction with eigenvalue $\sigma_{0}$, while $u_1$ and $u_2$ are eigenfunctions with eigenvalue $\sigma_{1}$. In particular, using that $\displaystyle\sum_{i=0}^{2}u_{i}^{2} = 1$, we conclude
\begin{align*}
0 &= \Delta_{g}\Big(\sum_{i=0}^{2}u_{i}^{2}\Big) = 2\sum_{i=0}^{2}u_{i}\Delta_{g}\,u_{i} + 2\sum_{i=0}^{2}\vert \nabla^{g}u_{i}\vert^{2}_{g} = - 4 + 2\sum_{i=0}^{2}\vert \nabla^{g}u_{i}\vert^{2}_{g}.
\end{align*}
Thus,
\begin{align}
\sum_{i=0}^{2}\vert \nabla^{g}u_{i}\vert^{2}_{g} &= 2.\label{eq:disk1}
\end{align}

Combining \eqref{eq:conformal.grad} and \eqref{eq:disk1}, it follows that $f\equiv 1$. Therefore, $(du)^{*}g_{\mathbb{B}^{2}_{r}} = g.$
\end{proof}

\section{A maximization problem}
\label{sec.max}

Let $\si$ be a compact orientable surface with  boundary and fix $r \in (0,\frac{\pi}{2})$. We define
\begin{align*}
\Theta_{r}^{\ast}(\si) &= \sup_{g \in \mathcal{M}(\si)} \Theta_{r}(\si,g),
\end{align*}
where $\Theta_{r}(\si,g)$ is the functional defined in \eqref{eq:funct_r}. For simplicity, given $u \in C^{\infty}(\pa\si)$ and $g \in \mathcal{M}(\si)$, we will also denote by $u$ its extension to $\si$ satisfying $\Delta_g u + 2u = 0$.

We would like to characterize the metrics realizing $\Theta_{r}^{\ast}(\si)$. The next proposition shows that a maximizing sequence of metrics cannot converge smoothly to the boundary of $\mathcal{M}(\Sigma)$.

\begin{proposition}
Let $\{g_{k}\}_{k \in \mathbb{N}}$ be a sequence in $\mathcal{M}(\si)$ converging smoothly to $g$ such that $\lambda_{0}^{D}(g) = 2$. Then, $\sigma_0(g_k) \to -\infty$.  
\end{proposition}

\begin{proof}
Suppose, by contradiction, there exists $c>0$ such that $\sigma_{0}(g_k) \geq -c$,  for all $k$. Then, up to a subsequence, $\sigma_{0}(g_k) \to \sigma$, where $\sigma \geq -c$.

Let $v_k$ be a positive $\sigma_{0}(g_k)$-eigenfunction. Then, 
\begin{equation}\label{eq:eing_gk}
\begin{split}
&\Delta_{g_k}v_k + 2 v_k = 0,  \ \textrm{in}\ \si,\\
&\frac{\pa v_k}{\pa\nu} = \sigma_{0}(g_k)v_k,  \ \textrm{on}\ \partial\si,
\end{split}
\end{equation}
and, by the Hopf maximum principle, $\sigma_{0}(g_k) < 0$; in particular,  $\sigma \leq 0$.

We can assume, without loss of generality, that $\int_{\si}v_{k}^{2}dA_{g_k} = 1$. Using \eqref{eq:eing_gk} and integrating by parts, we obtain
\begin{align*}
\int_{\si} |\nabla^{g_k} v_k|^{2}_{g_k}\,dA_{g_k} = 2\int_{\si} v^{2}_{k}\,dA_{g_k} + \sigma_{0}(g_k)\int_{\partial\si} v^{2}_{k}\,dL_{g_k} < 2.
\end{align*}
Thus, using the fact that $g_k$ converges smoothly to $g$, we get that $\{v_k\}_{k \in \mathbb{N}}$ is bounded in $W^{1,2}(\si,g)$. By the Rellich-Kodrachov Theorem and the compactness of the trace operator, we conclude that, up to a subsequence, $v_k \rightharpoonup v$ weakly in $W ^{1,2}(\si, g)$, $v_k \rightarrow v$ strongly in $L^{2}(\si, g)$ and $v_k \rightarrow v$ strongly in $L^{2}(\pa\si, g)$. 

By \eqref{eq:eing_gk}, for any $\phi \in C^{\infty}(\si)$ we have
$$\int_{\si} \big[g_{k}(\nabla^{g_k}v_k,\nabla^{g_k}\phi) - 2v_{k}\phi\big] dA_{g_k} = \sigma_{0}(g_k)\int_{\pa\si}v_{k}\phi\, dL_{g_k}.$$
Using the weak convergence $v_k \rightharpoonup v$ in $W^{1,2}(\Sigma, g)$ and the fact $g_k \to g$ smoothly, we obtain
$$\int_{\si} \big[g(\nabla^{g}v,\nabla^{g}\phi) - 2v\phi\big] dA_{g} = \sigma \int_{\pa\si}v\phi\, dL_{g}.$$
Hence, $v$ is a weak solution of
\begin{equation*}
\begin{split}
\Delta_{g}v + 2 v &= 0,\ \textrm{in}\ \si,\\
\frac{\pa v}{\pa\nu} &= \sigma v,\ \textrm{on}\ \partial\si,
\end{split}
\end{equation*}
and, by elliptic regularity, it is smooth on $\si$ (including the boundary). 

Since $v_k > 0$, we have $v \geq 0$. Observe that $\vert\vert v\vert\vert_{L^{2}(\si, g)} =1$, since $v_k\to v$ strongly in $L^2(\si, g)$ and $\vert\vert v_k\vert\vert_{L^{2}(\si, g)} =1$. In particular, $v$ can not be identically zero. Moreover, since $\Delta_{g}v = -2 v$, $v$ can not be constant. Hence, by the Hopf maximum principle, we get $\sigma < 0$ and $v > 0$ in $\si$.

Finally, since $\lambda_{0}^{D}(g) = 2$, we can consider a first eigenfunction $\psi$, i.e.,
\begin{equation*}
\begin{split}
\Delta_{g}\psi + 2\psi &= 0,\ \textrm{in}\ \si,\\
\psi &= 0,\ \textrm{on}\ \partial\si,
\end{split}
\end{equation*}
with $\psi > 0$ in $\si\backslash\pa\si$. We can write $\psi = hv$, where $h\vert_{\pa\si} \equiv 0$ and $h > 0$ in $\si\backslash\pa\si$. Applying the divergence theorem to the vector field $h^{2}v\nabla^{g}v$, we get
$$0 = \int_{\si}\big(|\nabla^{g} \psi|^{2}_{g}- 2\psi^{2}\big) dA_{g} = \int_{\si}v^{2}|\nabla^{g} h|^{2}_{g}\,dA_{g} \geq 0.$$
Thus, $h \equiv 0$, which is a contradiction. Therefore, we conclude that $\sigma_{0}(g_k) \to -\infty$.
\end{proof}

Now we will study the regularity of $\Theta_{r}(\si,g)$ with respect to $g$.

\begin{lemma}\label{lem:lips}
Let $g \in \mathcal{M}(\si)$ and consider a smooth path of metrics $t \mapsto g(t)$ such that $g(0) = g$ and $g(t) \in \mathcal{M}(\si)$ for all $ t\in [-\eta,\eta]$. Then $t \mapsto \sigma_{1}\big(g(t)\big)$ is Lipschitz in $[-\eta,\eta]$.
\end{lemma}

\begin{proof}
Denote $\sigma_i(t)=\sigma_i(g(t)), i=0,1,$ and assume they are uniformly bounded; that is, there exists a constant $C>0$ (that does not depend on $t$) such that $\vert\sigma_i(g(t))\vert\leq C$ for all $t\in  [-\eta,\eta]$ - we will show this claim at the end of the proof.

Fix $t_1, t_2 \in   [-\eta,\eta]$ and take $u_0$ and $u_1$ two $g(t_1)$-eigenfunctions with eigenvalues $\sigma_0(t_1)$ and $\sigma_1(t_1)$, respectively. Consider $V=\mbox{span}\{u_0, u_1\}$. This is a subspace of dimension 2, so we can find at least one $v\in V$ such that $v\neq0$ is $L^2(\partial\si, g(t_2))$-orthogonal to a first $\sigma_0(g(t_2))$-eigenfunction. 

Define 
$$
\mathcal {Q}_{g(t)}(v, v)=\frac{\displaystyle\int_{\si} |\nabla^{g(t)} v|^{2}_{g(t)}\,dA_{g(t)} - 2\int_{\si} v^{2}\,dA_{g(t)}}{\displaystyle\int_{\partial\si} v^2\,dL_{g(t)}}.
$$
We can write $\mathcal {Q}_{g(t_2)}(v, v)$ as
\begin{equation*}
\mathcal {Q}_{g(t_2)}(v, v)=\frac{\displaystyle\int_{\partial\si} v^2\,dL_{g(t_1)}}{\displaystyle\int_{\partial\si} v^2\,dL_{g(t_2)}}\cdot\frac{\displaystyle\int_{\si} |\nabla^{g(t_2)} v|^{2}_{g(t_2)}\,dA_{g(t_2)} - 2\int_{\si} v^{2}\,dA_{g(t_2)}}{\displaystyle\int_{\partial\si} v^2\,dL_{g(t_1)}}.
\label{lips:eq2}
\end{equation*}

Since $\{g(t)\}$ is a smooth path of metrics on a compact manifold, we know that there is a constant $K_0 > 0$ (which does not depend on $t_1, t_2$) such that
\begin{equation*}
\vert g(t_2)-g(t_1)\vert_{g}\leq K_0\vert t_2-t_1\vert,
\label{lips:metric}
\end{equation*}
and the metrics are uniformly equivalent.

Hence, there exist $K_1 > 0$, $K_2 > 0$ and  $K_3 > 0$ (which do not depend on $t_1, t_2$) such that
\begin{align}
\int_{\si} |\nabla^{g(t_2)} v|^{2}_{g(t_2)}\,dA_{g(t_2)}
&\leq (1+K_1\vert t_2-t_1\vert)\displaystyle\int_{\si} |\nabla^{g(t_1)} v|^{2}_{g(t_1)}\,dA_{g(t_1)},
\label{lips:eq17}\\
\int_{\si} v^{2}\,dA_{g(t_2)}&\geq (1-K_2\vert t_2-t_1\vert)\int_{\si} v^{2}\,dA_{g(t_1)},
\label{lips:eq15}\\
\int_{\partial\si} v^{2}\,dL_{g(t_1)}&\leq (1+K_3\vert t_1-t_2\vert)\int_{\partial\si} v^{2}\,dL_{g(t_2)}.\label{lips:eq13}
\end{align}

Inequalities \eqref{lips:eq17} and \eqref{lips:eq15} imply that for some $K > 0$ we have
\begin{align}
\nonumber\mathcal {Q}_{g(t_2)}(v, v)&\leq \frac{\displaystyle\int_{\partial\si} v^2\,dL_{g(t_1)}}{\displaystyle\int_{\partial\si} v^2\,dL_{g(t_2)}}\cdot \frac{\displaystyle\int_{\si} |\nabla^{g(t_1)} v|^{2}_{g(t_1)}\,dA_{g(t_1)} - 2\int_{\si} v^{2}\,dA_{g(t_1)}}{\displaystyle\int_{\partial\si} v^2\,dL_{g(t_1)}}\\
&+\frac{\displaystyle\int_{\partial\si} v^2\,dL_{g(t_1)}}{\displaystyle\int_{\partial\si} v^2\,dL_{g(t_2)}}\cdot\frac{K\vert t_2-t_1\vert\left(\displaystyle\int_{\si} |\nabla^{g(t_1)} v|^{2}_{g(t_1)}\,dA_{g(t_1)} +\int_{\si} v^{2}\,dA_{g(t_1)}\right)}{\displaystyle\int_{\partial\si} v^2\,dL_{g(t_1)}}.\label{lips:eq3}
\end{align}

Note that since $v\in V$, there are constants $a,b,$ such that  $v=au_0+bu_1$. In particular,
\begin{align*}
&\Delta_{g(t_1)}v+2v=0 \, \mbox{ in } \, \si,\\
& \frac{\partial v}{\partial\nu_{g(t_1)}}=a\sigma_0(t_1)u_0+b\sigma_1(t_1)u_1 \, \mbox{ on }\, \partial\si.
\end{align*}
Using the fact that $u_0$ and $u_1$ are $L^2(\partial\si, g(t_1))$-orthogonal (since they are eigenfunctions with different eigenvalues), we obtain
\begin{align}
\nonumber\displaystyle\int_{\si} |\nabla^{g(t_1)} v|^{2}_{g(t_1)}\,dA_{g(t_1)} &=-\int_{\si}v\Delta_{g(t_1)}vdA_{g(t_1)}+\int_{\partial\si}v\frac{\partial v}{\partial\nu}dL_{g(t_1)}\\
&=2\int_{\si}v^2dA_{g(t_1)}+a^2\sigma_0(t_1)\int_{\partial\si}u_0^2dL_{g(t_1)}+b^2\sigma_1(t_1)\int_{\partial\si}u_1^2dL_{g(t_1)}\label{lips:eq5}
\end{align}
and 
\begin{align}
\displaystyle\int_{\partial\si} v^{2}\,dL_{g(t_1)} &=a^2\int_{\partial\si}u_0^2dL_{g(t_1)}+b^2\int_{\partial\si}u_1^2dL_{g(t_1)}.
\label{lips:eq6}
\end{align}
Thus, \eqref{lips:eq5} and \eqref{lips:eq6} give that
\begin{equation}
\displaystyle\int_{\si} |\nabla^{g(t_1)} v|^{2}_{g(t_1)}\,dA_{g(t_1)}\leq C_1\left(\int_{\si} v^{2}\,dA_{g(t_1)}+ \int_{\partial\si} v^{2}\,dL_{g(t_1)}\right),\label{lips:eq10}
\end{equation}
where we are using that the eigenvalues are uniformly bounded from above. (In this proof all constants $C_i$'s do not depend on $t_1$ or $t_2$ and this remark will be omitted).

By Lemma 2.3 in \cite{AM2}, we can conclude that, for any $\delta>0$, there is a constant $C_\delta>0$ such that, for any $u \in C^\infty(\si),$ it holds
\begin{align}
\int_{\partial\si}u^2dL_{g}\leq \delta\int_{\si} |\nabla^{g} u|^{2}_{g}\,dA_{g}+C_\delta \int_{\si}u^2dA_{g}.\label{lips:eq7}\tag{I}
\end{align} 
Moreover, as in the proof of Proposition 3.3 in \cite{AM2}, for any $\delta>0$ there exists $C'_\delta>0$ such that for any $u \in C^\infty(\si)$ satisfying $\Delta_{g(t)}u+2u=0$ for some $t\in[-\eta,\eta],$ it holds that
\begin{align}
\int_{\si}u^2dA_{g}\leq  \delta\int_{\si} |\nabla^{g} u|^{2}_{g}\,dA_{g}+C'_{\delta} \int_{\partial\si}u^2dL_{g}.\label{lips:eq8}\tag{II}
\end{align}

Using \eqref{lips:eq7} and the fact that the metrics $g(t)$ are uniformly equivalent, we get
\begin{align}
 \nonumber\int_{\partial\si} v^{2}\,dL_{g(t_1)}&\leq C_2\int_{\partial\si} v^{2}\,dL_{g}\leq C_2\left(\delta\int_{\si} |\nabla^{g} v|^{2}_{g}\,dA_{g}+C_\delta \int_{\si}v^2dA_{g}\right)\\
&\leq C_2\left(C_3\delta\int_{\si} |\nabla^{g(t_1)} v|^{2}_{g(t_1)}\,dA_{g(t_1)}+C_4C_\delta \int_{\si}v^2dA_{g(t_1)}\right)\label{lips:eq9}.
\end{align}
Choosing $\delta>0$ sufficiently small and applying \eqref{lips:eq9} into \eqref{lips:eq10}, we obtain 
\begin{equation}
\displaystyle\int_{\si} |\nabla^{g(t_1)} v|^{2}_{g(t_1)}\,dA_{g(t_1)}\leq C_5\displaystyle\int_{\si}v^{2}\,dA_{g(t_1)}.
\label{lips:eq11}
\end{equation}

Now using \eqref{lips:eq8} and the fact that the metrics $g(t)$ are uniformly equivalent, we get
\begin{align*}
\nonumber C_6^{-1}\int_{\si}v^{2}\,dA_{g(t_1)}&\leq \int_{\si}v^{2}\,dA_{g}\leq \delta\int_{\si} |\nabla^{g} v|^{2}_{g}\,dA_{g}+C'_\delta \int_{\partial\si}v^2dL_{g}\\
\nonumber&\leq \delta C_7\int_{\si} |\nabla^{g(t_1)} v|^{2}_{g(t_1)}\,dA_{g(t_1)}+C'_\delta C_8 \int_{\partial\si}v^2dL_{g(t_1)}\\
&\leq  \delta C_7C_5\int_{\si} v^{2}\,dA_{g(t_1)}+C'_\delta C_8 \int_{\partial\si}v^2dL_{g(t_1)}.
\end{align*}

Choosing $\delta>0$ sufficiently small, we are able to conclude that
\begin{equation}
\displaystyle\int_{\si} v^{2}\,dA_{g(t_1)}\leq C_9\displaystyle\int_{\partial\si}v^{2}\,dL_{g(t_1)}.
\label{lips:eq12}
\end{equation}

Since
$\displaystyle \displaystyle\int_{\partial\si} v^2\,dL_{g(t_1)}\leq C_{10}\int_{\partial\si}  v^{2}\,dL_{g(t_2)},$  putting \eqref{lips:eq13}, \eqref{lips:eq11}, \eqref{lips:eq12} into \eqref{lips:eq3}, we get
\begin{align*}
\nonumber\mathcal {Q}_{g(t_2)}(v, v)&\leq  \frac{\displaystyle\int_{\si} |\nabla^{g(t_1)} v|^{2}_{g(t_1)}\,dA_{g(t_1)} - 2\int_{\si} v^{2}\,dA_{g(t_1)}}{\displaystyle\int_{\partial\si} v^2\,dL_{g(t_1)}}\\
\nonumber&\quad+K_3\vert t_1-t_2\vert\cdot  \frac{\displaystyle\int_{\si} |\nabla^{g(t_1)} v|^{2}_{g(t_1)}\,dA_{g(t_1)} - 2\int_{\si} v^{2}\,dA_{g(t_1)}}{\displaystyle\int_{\partial\si} v^2\,dL_{g(t_1)}}\\
\nonumber&\quad+C_{10}K\vert t_2-t_1\vert\cdot\frac{\left(\displaystyle C_5\int_{\si}v^2\,dA_{g(t_1)} +\int_{\si} v^{2}\,dA_{g(t_1)}\right)}{\displaystyle C_9^{-1}\int_{\si} v^2\,dA_{g(t_1)}}\\
&=\mathcal {Q}_{g(t_1)}(v, v)+K_3\vert t_1-t_2\vert\mathcal {Q}_{g(t_1)}(v, v)+\tilde{K}\vert t_1-t_2\vert.
\end{align*}

It follows from equalities \eqref{lips:eq5} and \eqref{lips:eq6} that $\mathcal {Q}_{g(t_1)}(v, v) \leq \sigma_1(t_1)$. Moreover, since $\sigma_i(t)$ is uniformly bounded, we also know that $\mathcal {Q}_{g(t_1)}(v, v)\leq C$. Hence, by the orthogonality of $v$ to the $\sigma_0(g(t_2))$-eigenspace, we get
\begin{align*}
\sigma_1(t_2)&\leq \mathcal {Q}_{g(t_2)}(v, v)\\
&\leq \mathcal {Q}_{g(t_1)}(v, v)+K_3\vert t_1-t_2\vert\mathcal {Q}_{g(t_1)}(v, v)+\tilde{K}\vert t_1-t_2\vert\\
&\leq \sigma_1(t_1)+\bar{K}\vert t_1-t_2\vert.
\end{align*}
That is, 
$$
\sigma_1(t_2)-\sigma_1(t_1)\leq \bar{K}\vert t_1-t_2\vert,
$$
where $\bar{K}$ is a constant that does not depend on $t_1, t_2.$

Applying the same arguments, reversing the roles of $t_1$ and $t_2$, we will obtain the other inequality to conclude that
$$
\vert \sigma_1(t_2)-\sigma_1(t_1)\vert\leq C\vert t_1-t_2\vert.
$$

Now let us prove that in fact the eigenvalues $\sigma_i(t)$ are uniformly bounded, $i=0,1.$

Let $W$ be a subspace of dimension 2 such that $W\setminus\{0\}\subset\{w\in C^\infty(\si); w\not\equiv0 \mbox{ on } \partial\si\}$. Hence, we can find at least one test function in $W$ for bounding $\sigma_1(t)$. Then,
\begin{align*}
\sigma_1(t)&\leq \sup_{w\in W\setminus\{0\}} \frac{\displaystyle\int_{\si} |\nabla^{g(t)} w|^{2}_{g(t)}\,dA_{g(t)} - 2\int_{\si} w^{2}\,dA_{g(t)}}{\displaystyle\int_{\partial\si} w^2\,dL_{g(t)}}\\
&\leq \sup_{w\in W\setminus\{0\}} \frac{\displaystyle\int_{\si} |\nabla^{g(t)} w|^{2}_{g(t)}\,dA_{g(t)}}{\displaystyle\int_{\partial\si} w^2\,dL_{g(t)}}\leq c_0 \sup_{w\in W\setminus\{0\}} \frac{\displaystyle\int_{\si} |\nabla^{g} w|^{2}_{g}\,dA_{g}}{\displaystyle\int_{\partial\si} w^2\,dL_{g}}=c_1.
\end{align*}

On the other hand, using the uniform equivalence of the metrics and inequality \eqref{lips:eq8}, we obtain for any $v\in C^{\infty}(\si)$ with $v\not\equiv 0$ on $\partial\si$ and $\Delta_{g(t)}v+2v=0$, 
\begin{align*}
 \frac{\displaystyle\int_{\si} |\nabla^{g(t)} v|^{2}_{g(t)}\,dA_{g(t)} - 2\int_{\si} v^{2}\,dA_{g(t)}}{\displaystyle\int_{\partial\si} v^2\,dL_{g(t)}}&\geq  \frac{\displaystyle c_1\int_{\si} |\nabla^{g} v|^{2}_{g}\,dA_{g} - c_2\int_{\si} v^{2}\,dA_{g}}{\displaystyle\int_{\partial\si} v^2\,dL_{g(t)}}\\
&\geq  \frac{\displaystyle (c_1-c_2\delta)\int_{\si} |\nabla^{g} v|^{2}_{g}\,dA_{g} - c_2c'_\delta\int_{\partial\si} v^{2}\,dL_{g}}{\displaystyle\int_{\partial\si} v^2\,dL_{g(t)}}\\
&\geq  \frac{\displaystyle - c_2c'_\delta\int_{\partial\si} v^{2}\,dL_{g}}{\displaystyle\int_{\partial\si} v^2\,dL_{g(t)}} \geq -c_3,
\end{align*}
where in the third inequality we chose $\delta>0$ sufficiently small so that $c_1-c_2\delta\geq 0$, and $c_3$ is a constant that does not depend on $t$. In particular, applying this when $v$ is a $\sigma_0(t)$-eigenfunction, we get
$$
\sigma_0(t)\geq -c_3.
$$
\end{proof}

\begin{remark}
The same arguments contained in the proof of Lemma \ref{lem:lips} can be applied, by choosing $V=\mbox{span}\{u_0, u_1, \ldots, u_k\}$ and $W$ with dimension $k+1,$ to show that the eigenvalues $\sigma_j(t)$  are Lipschitz for any $j$; and, moreover, given $k,$ there is a constant $c_k$  not depending on $t$ such that $-c_k\leq \sigma_j(t)\leq c_k$, for any $j=0, 1, \ldots, k.$
\end{remark}

Fix $g \in \mathcal{M}(\si)$ and consider a smooth path of metrics $g(t)$ as in the conclusion of the previous lemma. Choose a smooth path of functions $u_i(t),\,i=0,1$, such that $u_0(t)$ is an eigenfunction associated to $\sigma_0(t)$, $u_1(0)=u_1$ is an eigenfunction associated to $\sigma_{1}(0)$, and $u_1(t)$ is orthogonal to $u_0(t)$ for all $t$. Denote $h=\displaystyle\frac{dg}{dt}\Big\vert_{t=0}$, $v_1 = u_{1}'(0)$ and define
\begin{equation}
f_{1}(t) = \int_{\si} \big|\nabla^{g(t)}u_{1}(t)\big|^{2}_{g(t)}\,dA_{g(t)} -  2\int_{\si} u_{1}^{2}(t)\,dA_{g(t)} - \sigma_{1}(t) \int_{\partial\si} u_{1}^{2}(t)\,dL_{g(t)}.
\label{def_f}
\end{equation}

Recall that
\begin{align}\label{eq:der.volform}
\frac{d}{dt}\bigg\vert_{t=0}dA_{g(t)} = \frac{1}{2}\langle g,h\rangle\, dA_{g}, \quad \frac{d}{dt}\bigg\vert_{t=0}dL_{g(t)} = \frac{1}{2}h(T,T)\,dL_{g},
\end{align}
where $T$ is a $g$-unit tangent vector field to $\partial\si$.

Observe that by the definition of the eigenvalue, we have $f_1(t)\geq 0$ for any $t$, and $f_1(0)=0.$ Hence, if we assume that $\Theta_{r}(\si,g(t))$ is differentiable at $t=0$, then, in particular, $\sigma_1(t)$ is also differentiable at $t=0$ and $f_1$ has a critical point at $t=0.$ Thus,
\begin{align*}
0 = f_{1}'(0) &= 2\int_{\si} \Big[g\big(\nabla^{g} u_1,\nabla^{g} v_1\big) -  2 u_1 v_1\Big]dA_{g}  - 2\,\sigma_{1}(0) \int_{\partial\si} u_{1}v_1\,dL_{g}\\
&\quad + \int_{\si} \left\langle \frac{1}{2}\Big(\big|\nabla^{g} u_{1}\big|^{2}_{g}  -  2 u_{1}^{2}\Big)g - du_1\otimes du_1,h\right\rangle dA_{g} - \frac{1}{2}\,\sigma_{1}(0) \int_{\partial\si} u_{1}^{2}h(T,T)\,dL_{g}\\
&\quad - \big(\sigma_{1}\big)'(0) \int_{\partial\si} u_{1}^{2}(t)\,dL_{g},
\end{align*}
Using the fact that $u_1(0)$ is an eigenfunction associated to $\sigma_{1}(0)$, we have 
$$\int_{\si} g\big(\nabla^{g} u_1,\nabla^{g} v_1\big)\,dA_{g} - 2\int_{\si}  u_1 v_1 dA_{g}  = \sigma_{1}(0) \int_{\partial\si} u_{1}v_1\,dL_{g}.$$
Hence, choosing $u_1(0)$ such that $\vert\vert u_1\vert\vert_{L^{2}(\partial\si,g)} = 1$, it follows 
\begin{equation}\label{deriv_sigma}
\big(\sigma_{1}\big)'(0) = \int_{\si} \left\langle \frac{1}{2}\Big(\big|\nabla^{g} u_{1}\big|^{2}_{g}  -  2 u_{1}^{2}\Big)g - du_1\otimes du_1,h\right\rangle dA_{g} - \frac{1}{2}\,\sigma_{1}(g) \int_{\partial\si} u_{1}^{2}h(T,T)\,dL_{g}.
\end{equation}

\begin{remark}
Since the eigenvalue $\sigma_0$ is simple, one can use eigenvalue perturbation theory to show that $t \mapsto \sigma_{0}(t)$ is differentiable for $t$ small. Defining a function $f_{0}(t)$ exactly as we did for $f_{1}(t)$ in {\rm(\ref{def_f})} $-$ using $u_0(t)$ instead of $u_1(t)$ $-$ and doing the same computations as above, we will get an equivalent formula for $(\sigma_0)'(0)$ as in {\rm(\ref{deriv_sigma})}. 
\end{remark}

Therefore, if $\Theta_{r}(t) := \Theta_{r}\big(\si,g(t)\big)$ is differentiable at $t=0$, using \eqref{eq:der.volform}, \eqref{deriv_sigma}, and the previous remark, we obtain
\begin{align*}
\Theta_{r}'(0) = Q_h(u_1) &:= \int_{\si} \Big\langle \tau\big(\cos r \, u_0\big)  +\tau\big(\sin r \, u_1\big) +  \, g,h\Big\rangle\, dA_{g} \\ 
&\quad\ + \int_{\partial\si} F\big(\cos r \, u_0,\sin r \, u_1\big) h(T,T)\,dL_{g} ,
\end{align*}
where
\begin{align*}
\tau(v) &= \frac{\vert\partial\si\vert_{g}}{2}\big(|\nabla^{g} v|^{2}_{g} -  2 v^{2}\big)g - \vert\partial\si\vert_{g}\, dv\otimes dv,\\
F\big(v_{0},v_{1}\big) &= \frac{\sigma_{0}(g)}{2}\big(\cos^{2} r - |\partial\si|_{g} v_{0}^{2}\big) + \frac{\sigma_{1}(g)}{2}\big(\sin^{2} r - |\partial\si|_{g} v_{1}^{2}\big).\label{eq:def.F}
\end{align*}

Let $V_{k}(g)$ be the eigenspace associated to $\sigma_{k}(g), \, k=0,1$. In the following, $u_0$ denotes the unique positive $\sigma_{0}(g)$-eigenfunction such that $\vert\vert u_0\vert\vert_{L^{2}(\partial\si,g)} = 1$. Also, consider the Hilbert space
$$\mathcal{H}_{g} := L^2\big(S^2(\si),g\big)\times L^2(\partial\si,g),$$
where $S^2(\si)$ is the space of symmetric $(0,2)$-tensor fields on $\si$.

\begin{lemma}\label{lem:HB}
Suppose $g \in \mathcal{M}(\si)$ satisfies $\Theta_{r}(\si,g) = \Theta_{r}^{\ast}(\si)$. Then, for any $(h,\psi) \in \mathcal{H}_{g}$, there exists $u_1 \in V_1(g)$ such that $\vert\vert u_1\vert\vert_{L^{2}(\partial\si,g)} = 1$ and 
$$\Big\langle(h,\psi),\Big(\tau\big(\cos r \, u_0\big)  + \tau\big(\sin r \, u_1\big) +  \,g,F\big(\cos r \, u_{0},\sin r \, u_{1}\big)\Big)\Big\rangle_{\mathcal{H}_{g}} = 0.$$
\end{lemma}

\begin{proof}
Consider $(h,\psi) \in \mathcal{H}_{g}$. Since $C^\infty\big(S^2(\si)\big) \times C^\infty(\partial\si)$ is dense in $\mathcal{H}_{g}$, for each $k\in\mathbb{N}$ there exists a smooth pair $(\tilde{h}_k,\tilde{\psi}_k)$ such that $(\tilde{h}_k,\tilde{\psi}_k) \rightarrow (h,\psi)$ in $L^2$. Moreover, for each $k$, we may redefine $\tilde{h}_k$ in a neighborhood of the boundary to a smooth tensor field $h_{k}$ with $h_k(T,T)\vert_{\p \si} = \psi$, and such that the change in the $L^2$ norm is arbitrarily small. Thus, we obtain a smooth sequence $h_k$ such that $\big(h_k,h_k(T,T)\big) \rightarrow (h,\psi)$ in $L^2$.

Let $g(t)= g + th_k$. 
Then $g_{0}=g$ and $\displaystyle\frac{dg}{dt}\Big|_{t=0} =h_k$.
Since $g$ maximizes $\Theta_{r}$, given any $\varepsilon >0$, we have 
\[\int_{-\varepsilon}^0 \Theta_{r}^{\prime}(t)\,dt = \Theta_{r}(0) - \Theta_{r}(-\epsilon) \geq 0,
\]
Therefore, there is $t \in (-\varepsilon,0)$ such that 
$\Theta_{r}^{\prime}(t)$ exists and is nonnegative. Let $t_j$ be such a sequence of points with $t_j<0$ and $t_j \rightarrow 0$ such that $\Theta_{r}^{\prime}(t_j) \geq 0$. Choose $u_j \in V_{1}\big(g(t_j)\big)$ with 
$||u_j||_{L^2(\p \si)}=1$. We can use standard elliptic boundary value estimates to obtain bounds on $u_j$ and its derivatives up to the boundary. Thus, after passing to a subsequence, $u_j$ converges in $C^2(\si)$ to an eigenfunction $u_{k}^{-} \in V_{1}(g)$ with $||u_{k}^{-}||_{L^2(\pa\si,g)}=1$ and, consequently, it follows that $Q_{h_k}(u_{k}^{-}) \geq 0$. By a similar argument, taking a limit from the right, there exists $u_{k}^{+} \in V_{1}(g)$ with $||u_{k}^{+}||_{L^2(\pa\si,g)}=1$ such that $Q_{h_k}(u_{k}^{+}) \leq 0$.

As above, after passing to subsequences, $u_{k}^{+} \rightarrow u_{+}$ and $u_{k}^{-} \rightarrow u_{-}$ in $C^2(\si)$, and
\begin{align*}
& \left\langle (h,\psi), \Big(\tau\big(\cos r \, u_0\big)  + \tau\big(\sin r\, u^{+}\big) + \, g,F\big(\cos r\, u_{0},\sin r\, u^{+}\big)\Big) \right\rangle_{\mathcal{H}_{g}} =& \lim_{i \rightarrow \infty} Q_{h_k}(u_{k}^{+})\leq 0 \\
& \left\langle (h,\psi), \Big(\tau\big(\cos r\, u_0\big)  + \tau\big(\sin r\, u^{-}\big) +  \, g,F\big(\cos r\, u_{0},\sin r\, u^{-}\big)\Big) \right\rangle_{\mathcal{H}_{g}} =& \lim_{k \rightarrow \infty} Q_{h_k}(u_{k}^{-})\geq 0.
\end{align*}
Therefore, the lemma follows.
\end{proof}

\begin{thm2}\label{prop:char.min} 
Let $\si$ be a compact surface with boundary.
\begin{itemize}
\item[(i)] Suppose $g \in \mathcal{M}(\si)$ satisfies $\Theta_{r}(\si,g) = \Theta_{r}^{\ast}(\si)$. Then, there exist independent eigenfunctions $v_0 \in V_{0}(g)$ and $v_1,\ldots,v_n \in V_{1}(g)$ which induce a free boundary  minimal isometric immersion $v = (v_0,v_1,\ldots,v_n):(\si,g) \to \ba$.
\item[(ii)] Suppose $g \in \mathcal{M}(\si)\cap \mathfrak{C}$ satisfies $\Theta_{r}(\si,g) = \Theta_{r}(\si,\mathfrak{C})$. Then, there exist independent eigenfunctions $v_0 \in V_{0}(g)$ and $v_1,\ldots,v_n \in V_{1}(g)$ which induce a free boundary harmonic map $v = (v_0,v_1,\ldots,v_n):(\si,g) \to \ba$.
\end{itemize}
\end{thm2}

\begin{proof}
(i) Define $\mathcal{K}$ as the convex hull in $\mathcal{H}_{g}$ of the following subset
$$\left\{\Big(\tau\big(\cos r \, u_0\big)  + \tau\big(\sin r\, u_1\big) +  \, g,F\big(\cos r \, u_{0},\sin r \, u_{1}\big)\Big);\, u_{1} \in V_1(g),  ||u_{1}||_{L^2(\pa\si,g)}=1\right\}.$$

We claim that $(0,0) \in \mathcal{K}$. Suppose this is not true. Since $\mathcal{K}$ is a closed convex subset of $\mathcal{H}_{g}$, by the geometric form of the Hahn-Banach Theorem, there exists $(h,\psi) \in \mathcal{H}_{g}$ such that for any $u_{1} \in V_1(g)$ with $||u_{1}||_{L^2(\pa\si,g)}=1,$ we have
\begin{align*}
\bigg\langle\left(h,\psi\right),\Big(\tau\big(\cos r\, u_0\big)  + \tau\big(\sin r \, u_1\big) +  \,g,F\big(\cos r\, u_{0}, \sin r\, u_{1}\big)\Big)\bigg\rangle_{\mathcal{H}_{g}} < 0.
\end{align*}
However, this last inequality contradicts the previous lemma. 

Thus, there exist $u_1,\ldots,u_n \in V_{1}(g)$ with $||u_{i}||_{L^2(\pa\si,g)}=1$, and $t_1,\ldots,t_n\in \rr_{+}$ with $\displaystyle\sum_{j=1}^n t_j=1$ such that
\begin{align}
0 &= \sum_{j=1}^n t_j\Big[\tau\big(\cos r\, u_0\big)  + \tau\big(\sin r\, u_j\big)+  \,g\Big], \  \textrm{in}\ \si,
\label{eq:extremal1}\\
0 &=\sum_{j=1}^n t_j F\big(\cos r\, u_{0}, \sin r\, u_{j}\big),\ \textrm{on}\ \partial\si.\label{eq:extremal2}
\end{align}
We can rewrite (\ref{eq:extremal1}) as
$$
\tau\big(\cos r\, u_0\big) +  \,g+ \sum_{j=1}^n\tau\big(\sqrt{t_j}(\sin r)u_j\big) = 0.
$$

Denote $v_0 = \vert\partial\si\vert^{\frac{1}{2}} (\cos r)u_0$ and $v_j = \left(t_j\vert\partial\si\vert\right)^{\frac{1}{2}} (\sin r)u_{j},\, j=1,\ldots,n$. Hence, by the definition of $\tau$, we get
\begin{equation}
\label{eq:12}
\sum_{j=0}^{n} \left[dv_j\otimes dv_j - \frac{1}{2}\Big(\big|\nabla^{g} v_j\big|^{2} -  2\,v_{j}^{2}\Big)g - g\right] = 0.
\end{equation}
Since $\mathrm{tr}_{g}\,g=2$ and $\mathrm{tr}_{g}(dv_j\otimes dv_j)=\big|\nabla^{g} v_j\big|^{2},$ taking the trace in the above equality yields
\begin{equation}
\sum_{j=0}^n v_{j}^2 = 1.
\label{eq:sumv}
\end{equation}
In particular, using this fact back into \eqref{eq:12}, we conclude that
\begin{equation}
\sum_{j=0}^{n} dv_j\otimes dv_j = \frac{1}{2}\bigg(\sum_{j=0}^{n}\big|\nabla^{g} v_j\big|^{2}\bigg)g.
\label{eq:imers}
\end{equation}


Using that $\Delta_{g}\,v_j +  2 v_j = 0$ together with \eqref{eq:sumv}, we obtain
\begin{align*}
0&=\Delta_g \bigg(\sum_{j=0}^n v_{j}^2\bigg) = \sum_{j=0}^{n}\bigg(2v_j\Delta_g v_j+ 2\vert \nabla^g v_j\vert^2\bigg) =-4\sum_{j=0}^n v_j^2+ 2\sum_{j=0}^n\vert \nabla^g v_j\vert^2\\
&\implies \sum_{j=0}^n\vert \nabla^g v_j\vert^2=2.
\end{align*}

Therefore, by \eqref{eq:imers}, we conclude that
\begin{equation}
\sum_{j=0}^{n} dv_j\otimes dv_j =g,
\label{eq:iso}
\end{equation}
and $v = (v_0,v_1,\ldots,v_n)$ is an isometric minimal immersion of $\si$ into $\mathbb{S}^{n}$.

Now, we will prove that $v(\si) \subset \ba$.
Observe that by \eqref{eq:extremal2} along $\partial\si$ we have
\begin{align*}
0 &= \sum_{j=1}^n t_j F\left(\vert\partial\si\vert^{-\frac{1}{2}}v_0,\left(t_j\vert\partial\si\vert\right)^{-\frac{1}{2}}v_j\right)\\
&=\frac{\sigma_{0}}{2}(\cos^{2} r -v_{0}^{2})+\frac{\sigma_1}{2}\bigg(\sin^2 r-\sum_{j=1}^n v_{j}^{2}\bigg)\\
&=\frac12\sigma_{0}\cos^{2} r+\frac12\sigma_{1}\sin^{2}r-\frac12\bigg[\sum_{j=0}^{n}v_j\frac{\partial v_j}{\partial\nu}\bigg]\\
&=\frac12\sigma_{0}\cos^{2} r+\frac12\sigma_{1}\sin^{2}r-\frac14\frac{\partial}{\partial\nu}\bigg(\sum_{j=0}^{n}v_{j}^{2}\bigg)\\
&=\frac12\sigma_{0}\cos^{2} r+\frac12\sigma_{1}\sin^{2}r.
\end{align*}
Hence,
\begin{equation}\label{eq:bound.cond}
\sigma_{1}=-(\cot^2r)\sigma_0.
\end{equation}

By \eqref{var-sigma0}, we have $\sigma_0<0$. On the other hand,
\begin{align*}
0&=\frac12\frac{\partial}{\partial\nu}\bigg(\sum_{j=0}^{n}v_j^2\bigg)=\sigma_0v_0^2+\sigma_1\bigg(\sum_{j=1}^nv_j^2\bigg)
=\sigma_0v_0^2+\sigma_1(1-v_0^2),
\end{align*}
from where it follows that
\begin{equation}
\sigma_1=(\sigma_1-\sigma_0)v_0^2 \ \mbox{on} \ \partial\si.
\label{eq:v0}
\end{equation}
Thus, combining \eqref{eq:bound.cond} and \eqref{eq:v0}, and using the fact that $v_0$ is positive, we conclude
\begin{equation*}
v_{0} = \cos r \ \textrm{on}\ \partial\si.
\end{equation*}

Since $v_0$ is positive and $\Delta_{g}v_{0} +  2 v_{0} = 0$, its minimum is attained on the boundary; in particular,
\begin{equation*}
v_{0} \geq \cos r \ \textrm{in}\ \si.
\end{equation*}
Therefore, $v(\si) \subset \ba$.  

It remains to check the free boundary condition. Applying the tensor fields in \eqref{eq:iso} to $(\nu,\nu)$, we obtain that on $\partial \si$ it holds 
\begin{align*}
1= \sum_{j=0}^n\left(\frac{\partial v_j}{\partial \nu}\right)^2= \sigma_0^2v_0^2+\sigma_1^2\bigg(\sum_{j=1}^nv_j^2\bigg) =  \sigma_0^2\cos^2r+\sigma_0^2\cot^4r(1-\cos^2r),
\end{align*}
from where it follows that $\sigma_0=-\tan r$ (since $\sigma_0<0$) and, consequently, $\sigma_1=\cot r$. Therefore, by Proposition \ref{prop-eigen}, the immersion is free boundary.

\vspace{0.3cm}

(ii) Fix a conformal class $\mathfrak{C} = [g]$. Consider a smooth path $t \mapsto f(t)$ of positive functions such that $f(0) \equiv 1$ and $g(t) = f(t)g \in \mathcal{M}(\si)$ for all $t$. Denote $w = f^{\prime}(0)$ so that $\displaystyle\frac{dg}{dt}\Big\vert_{t=0} = wg$. Following the previous notation and using the previous calculations, we have
\begin{align*}
\big(\sigma_{k}\big)'(0) &= \int_{\si} \left\langle \frac{1}{2}\Big(\big|\nabla^{g} u_{k}\big|^{2}_{g}  -  2 u_{k}^{2}\Big)g - du_k\otimes du_k,wg\right\rangle dA_{g} - \frac{1}{2}\,\sigma_{1}(g) \int_{\partial\si} u_{k}^{2}\,w\,g(T,T)\,dL_{g},\\
&= -  2 \int_{\si}u_{k}^{2}\,w\, dA_{g} - \frac{1}{2}\,\sigma_{k}(g) \int_{\partial\si} u_{k}^{2}\,w\,dL_{g},
\end{align*}
hence,
\begin{align*}
\Theta_{r}'(0) = -2|\pa\si|_{g}\int_{\si} \big[(\cos^2 r)u_{0}^{2}+ (\sin^2 r)u_{1}^{2}\big]w\, dA_{g}+ \int_{\partial\si} F\big((\cos r)u_0,(\sin r)u_1\big)w\,dL_{g}.
\end{align*}

Consider $w \in L^{2}(\si,g)$ and $\psi \in L^{2}(\pa\si,g)$. Arguing as in the proof of Lemma \ref{lem:HB} one can prove that there exists $u_1 \in V_1(g)$ such that $\vert\vert u_1\vert\vert_{L^{2}(\partial\si,g)} = 1$ and 
$$\Big\langle(w,\psi),\Big(-2|\pa\si|_{g}(\cos^2 r)u_{0}^{2} - 2|\pa\si|_{g}(\sin^2 r)u_{1}^{2},F\big((\cos r)u_{0},(\sin r)u_{1}\big)\Big)\Big\rangle_{L^2} = 0.$$

Now, define $\mathcal{K}$ as the convex hull in $L^{2}(\si,g)\times L^{2}(\pa\si,g)$ of the following subset
$$\left\{\Big(-2|\pa\si|_{g}(\cos^2 r)u_{0}^{2} - 2|\pa\si|_{g}(\sin^2 r)u_{1}^{2},F\big((\cos r)u_{0},(\sin r)u_{1}\big)\Big);\, u_{1} \in V_1(g)\right\}.$$
Then, arguing as in the proof of (i), we conclude $(0,0) \in \mathcal{K}$. 

Thus, there exist $u_1,\ldots,u_n \in V_{1}(g)$, and $t_1,\ldots,t_n\in \rr$, such that $\sum_{j=1}^n t_j=1$ and
\begin{align}
0 &= \sum_{j=1}^n t_j\big[(\cos^2 r)u_{0}^{2} +(\sin^2 r)u_{j}^{2}\big], \  \textrm{in}\ \si,
\label{eq:extremal3}\\
0 &=\sum_{j=1}^n t_j F\big((\cos r)u_{0},(\sin r)u_{j}\big),\ \textrm{on}\ \partial\si.\label{eq:extremal4}
\end{align}

Denote $v_0 = (\cos r)u_0$ and $v_j = \sqrt{t_{j}}\,(\sin r)u_{j},\, j=1,\ldots,n$. Then, by \eqref{eq:extremal3}, we have
$$\sum_{j=0}^{n}v_{j}^{2} = 1,$$
and, hence, $v(\si) \subset \mathbb{S}^n$. Also, $v: (\si,\mathfrak{C}) \to \mathbb{S}^n$ is a harmonic map, since $\Delta_{g}\,v_j + 2v_j =0$, $j=0,1,\ldots,n$. Finally, using \eqref{eq:extremal4} and arguing as in the proof of (i), we conclude that $v(\si) \subset \ba$, $v(\pa\si) \subset \pa\ba$, and the map satisfies the free boundary condition.
\end{proof}

\section{A Family of free boundary minimal annuli in the spherical cap}\label{sec:rot.annuli}


In \cite{dCD} do Carmo and Dajczer described the parametrization of the family of rotational minimal hypersurfaces in the round sphere $\mathbb S^n$. For the case $n=3$, the parametrization is given by $\Phi_{a}:\rr\times \mathbb S^1\to \mathbb S^3$,
\begin{equation}
\begin{array}{rcl}\label{eq-annulus}
\Phi_{a}(s,\theta)&=&\left(\sqrt{\frac{1}{2}-a\cos(2s)}\cos\varphi(s), \sqrt{\frac{1}{2}-a\cos(2s)}\sin\varphi(s),  \right.\\
&&
 \left.\sqrt{\frac{1}{2}+a\cos(2s)}\cos\theta, \sqrt{\frac{1}{2}+a\cos(2s)}\sin\theta\right),
\end{array}
\end{equation}

where $-\frac{1}{2}<a\leq 0$ is a constant and $\varphi(s)$ is given by
$$
\varphi(s)=\sqrt{\frac{1}{4}-a^2}\displaystyle\int_0^s \frac{1}{(\frac{1}{2}-a\cos(2t))\sqrt{\frac{1}{2}+a\cos(2t)}}dt.
$$

See \cite{Ot} for a different parametrization.

As observed in \cite{LX} (see also \cite{RdO} for more details), it is possible to prove that for any $0<r\leq\frac{\pi}{2}$, there exist $-\frac12<a\leq0$ and $s_0\in\mathbb R$ such that $\Phi_a: [-s_0,s_0]\times\mathbb S^1\to  \mathbb{B}^3(r)$ is a free boundary minimal annulus in  $\mathbb{B}^3(r)$.


To simplify the notation, let us denote $\rho(s)=\sqrt{\frac{1}{2}-a\cos(2s)}$ and omit the subscrit  $a$ in the parametrization. Hence, we have
\begin{align*}
\Phi(s,\theta)&=(\rho(s)\cos\varphi(s),  \rho(s)\sin\varphi(s), \sqrt{1-\rho(s)^2}\cos\theta, \sqrt{1-\rho(s)^2}\sin\theta),
\nonumber\\
\Phi_s(s,\theta)&=\displaystyle(\rho^{\prime}(s)\cos\varphi(s)-\rho(s)\sin\varphi(s)\varphi'(s),
 \rho^{\prime}(s)\sin\varphi(s)+\rho(s)\cos\varphi(s)\varphi'(s),\nonumber\\
& \quad\,   -\frac{\rho(s)\rho^{\prime}(s)}{\sqrt{1-\rho(s)^2}}\cos\theta, -\frac{\rho(s)\rho^{\prime}(s)}{\sqrt{1-\rho(s)^2}}\sin\theta ),\\
\Phi_\theta(s,\theta)&=(0,0, -\sqrt{1-\rho(s)^2}\sin\theta, \sqrt{1-\rho(s)^2}\cos\theta).
\end{align*}
Thus, using that $\varphi'(s)=\displaystyle\frac{\sqrt{1/4-a^2}}{\rho(s)^2\sqrt{1-\rho(s)^2}}$ and $\rho(s)^2\rho^{\prime}(s)^2+1/4-a^2=\rho(s)^2(1-\rho(s)^2),$ we get
\begin{align*}
g_{11}&=\langle \Phi_s, \Phi_s\rangle=\displaystyle\frac{\rho^{\prime}(s)^2+\rho(s)^2(1-\rho(s)^2)\varphi'(s)^2}{1-\rho(s)^2}=1\\
 g_{12}&=\langle \Phi_s, \Phi_\theta \rangle=0 \, \mbox{ and } \,  g_{22}=\langle \Phi_\theta, \Phi_\theta\rangle= 1-\rho(s)^2.
\end{align*}

Then,
\begin{align*}
\Delta_g f&=\displaystyle\frac{1}{\sqrt{\det(g)}}\sum_{i=1}^{2}\partial_i(g^{ij}\sqrt{\det(g)}\partial_jf)\\
&=\displaystyle \partial^2_{ss}f+\frac{1}{1-\rho^2} \partial^2_{\theta\theta}f-\frac{\rho\rho^{\prime}}{1-\rho^2}\partial_sf.
\end{align*}
Hence, the equation $\Delta_gf+2f=0$ is equivalent to 
\begin{equation}\label{eq:min-edo}
\displaystyle \partial^2_{ss}f+\frac{1}{1-\rho^2} \partial^2_{\theta\theta}f-\frac{\rho\rho^{\prime}}{1-\rho^2}\partial_sf+2f=0.
\end{equation}


\begin{theorem}
Let $\Phi_{a}:[-s_0, s_0]\times\mathbb S^1\to \mathbb{B}^3(r)$ be a free boundary minimal annulus in $\mathbb{B}^3(r)$ given by $(\ref{eq-annulus})$, and denote $g = \Phi_{a}^{\ast}\,g_{\mathbb{B}^3(r)}$. Then, $\phi_0(s,\theta)=\rho(s)\cos\varphi(s)$ is a $\sigma_0(g)$-eigenfunction, while $\phi_1(s,\theta)=\rho(s)\sin\varphi(s),$ $\phi_2(s,\theta)=\sqrt{1-\rho(s)^2}\cos\theta$ and $\phi_3(s, \theta)= \sqrt{1-\rho(s)^2}\sin\theta$ are $\sigma_1(g)$-eigenfunctions. Moreover, $\Phi_a$ is an embedding.
\end{theorem}

\begin{proof}
By Proposition \ref{prop-eigen}, we already know that $\phi_i$ are eigenfunctions.
Since $\phi_0$ is positive, then $\phi_0$ is necessarily a $\sigma_0$-eigenfunction with $\sigma_0=-\tan(r)$. Now we would like to prove that $\phi_i$, $i=1,2,3$, are $\sigma_1$-eigenfunctions, that is, $\sigma_1=\cot(r).$ We will use the same arguments as in the proof of the disk case (see Theorem \ref{thm-disk}).

Let $f:[-s_0,s_0]\times\mathbb S^1\to\mathbb R$ be a $\sigma_1$-eigenfunction. We have
$$f(s,\theta)=a_0(s)\cdot 1+\displaystyle\sum_{k=1}^{\infty}[a_k(s)\cos(k\theta)+b_k(s)\sin(k\theta)]$$
and  $a_0(s)$, $a_k(s)\cos(k\theta)$, $b_k(s)\sin(k\theta)$ are $\sigma_1$-eigenfunctions as well. By Theorem \ref{thm-HS}, the functions $a_k$ and $b_k$ are identically zero for $k\geq2.$

Observe that $a_0(s)$ satisfies the following linear ODE of order 2:
$$
\displaystyle a_0{''}(s)-\frac{\rho(s)\rho^{\prime}(s)}{1-\rho^2(s)}a_0'(s )+2a_0(s)=0.
$$
and we can check that $\phi_0(s,\theta)=\rho(s)\cos\varphi(s)$ and $\phi_1(s,\theta)=\rho(s)\sin\varphi(s)$ are solutions to this ODE as well. Thus, $a_0(s)$ is necessarily a linear combination of them. Since $\phi_0$ is a $\sigma_0$-eigenfunction, and $a_0(s)$ and $\phi_1$ are eigenfunctions with eigenvalues different from $\sigma_0$, we know that $a_0(s)$ and $\phi_1$ are both $L^2$-orthogonal to $\phi_0$ on the boundary, from where we can conclude that $a_0(s)$ is a constant multiple of $\phi_1$. In particular, if $a_0(s)$ is not identically zero, we get that $\phi_1$ is a $\sigma_1$-eigenfunction and $\sigma_1=\cot(r)$ and, consequently, $\phi_2$ and $\phi_3$ are $\sigma_1$-eigenfunctions as well.

Now suppose that $a_0\equiv0$. 

Assume $a_1(s)$ is not identically zero. We know by \eqref{eq:min-edo} that $a_1(s)\cos(\theta)$ satisfies:
$$
\displaystyle\left[ a_1{''}(s)-\frac{\rho(s)\rho^{\prime}(s)}{1-\rho^2(s)}a_1'(s )+\left(2-\frac{1}{1-\rho^2(s)}\right)a_1(s)\right]\cos\theta=0.
$$
In particular, $a_1(s)$ satisfies the linear ODE of order 2:
\begin{equation}\label{eq:edo-a1}
\displaystyle\left[ a_1{''}(s)-\frac{\rho(s)\rho^{\prime}(s)}{1-\rho^2(s)}a_1'(s )+\left(2-\frac{1}{1-\rho^2(s)}\right)a_1(s)\right]=0.
\end{equation}

We can check that $\sqrt{1-\rho^2(s)}$ satisfies this ODE, now we need to find a second solution. 
Using that $\rho(s)\rho(s)'=a\sin(2s)$ and $1-2\rho^2(s)=2a\cos(2s),$ we can rewrite \eqref{eq:edo-a1} as
\begin{equation}\label{eq:edo-a11}
\displaystyle\left(\frac{1}{2}+a\cos(2s)\right)a_1''(s)-a\sin(2s)a_1'(s)+2a\cos(2s)a_1(s)=0.
\end{equation}

Since this is a second order linear ODE, we can write a second solution to it as $\sqrt{1-\rho^2(s)}h(s)$, for some function $h(s).$ Applying $\sqrt{1-\rho^2(s)}h(s)$ to \eqref{eq:edo-a11}, we get the function $h(s)$ must satisfy
$$
-3a\sin(2s)h'(s)+\left(\frac{1}{2}+a\cos(2s)\right)h''(s)=0,
$$
which implies, after possibly a rescaling, that
$$
h(s)=\int_0^s\left(\frac{1}{2}+a\cos(2t)\right)^{-\frac{3}{2}}dt +C,
$$
for some constant $C.$ We can suppose $C=0$ since we are just interested in finding a basis for the space of solutions to an ODE.


We have that $\sqrt{1-\rho^2(s)}$ and $\sqrt{1-\rho^2(s)}h(s)$ are two linearly independent solutions to the ODE given by \eqref{eq:edo-a11}, then any solution to this ODE is a linear combination of them.

Note that  the outward conormal $\frac{\partial}{\partial\nu}$ to the boundary of $[-s_0,s_0]\times\mathbb S^1$ satisfies $\frac{\partial}{\partial\nu}=-\partial_s$ on $\{-s_0\}\times\mathbb S^1$ and $\frac{\partial}{\partial\nu}=\partial_s$ on $\{s_0\}\times\mathbb S^1$. Hence, using the fact that $h(-s)=h(s),$ we can check that 
$$
\frac{\frac{\partial h}{\partial \nu}(-s_0)}{h(-s_0)}=\frac{\frac{\partial h}{\partial \nu}(s_0)}{h(s_0)}=:\mu.
$$

Thus, the functions  $\phi_2(s,\theta)=\sqrt{1-\rho^2(s)}\cos\theta$ and $f_1(s,\theta)=\sqrt{1-\rho^2(s)}h(s)\cos\theta$ are  Steklov eigenfunctions with eigenvalues $\cot(r)$ and $\cot(r)+\mu$, respectively.

%
%
%
%
%

Observe that 
$$
\cot(r)=\displaystyle\frac{\displaystyle\frac{\partial \phi_1}{\partial s}}{\phi_1}=\frac{-\rho\rho^{\prime}}{1-\rho^2}
$$
and
\begin{align*}
\cot(r)+\mu&=\displaystyle\frac{\displaystyle\frac{\partial f_1}{\partial s}}{f_1}=\displaystyle\frac{-\rho\rho^{\prime}\displaystyle\int_0^{s_0}\left(1-\rho^2\right)^{-\frac{3}{2}}dt+(1-\rho^2)^{-1/2}}{(1-\rho^2)\displaystyle\int_0^{s_0}\left(1-\rho^2\right)^{-\frac{3}{2}}dt}\\
&=\cot(r)+\displaystyle\frac{1}{(1-\rho^2)^{\frac{3}{2}}\displaystyle\int_0^{s_0}\left(1-\rho^2\right)^{-\frac{3}{2}}dt}.
\end{align*}
Hence, $\mu>0$. Therefore, $\phi_2(s,\theta)=\sqrt{1-\rho^2(s)}\cos\theta$ is a $\sigma_1$-eigenfunction with $\sigma_1=\cot(r)$ and, in particular, $\phi_3(s,\theta)=\sqrt{1-\rho^2(s)}\sin\theta$ and $\phi_1(s,\theta)=\rho(s)\sin\varphi(s)$ are  $\sigma_1$-eigenfunctions as well.

For the case where $a_1(s)$ is identically zero, we apply the same arguments to the function $b_1(s)$ and we get the same conclusions. Therefore, the annulus is indeed immersed by eigenfunctions as in the statement of theorem.

Now to see that $\Phi_a$ is an embedding, observe that since $\Phi_{a}([-s_0, s_0]\times\mathbb S^1)\subset\mathbb{B}^3(r)$, the first coordinate $\phi_0$ satisfies $\phi_0\geq \cos(r)\geq 0$  and, since it is minimal, we know that $\Delta \phi_0+2\phi_0=0$. Hence, we can apply the Hopf Lemma to conclude that $\displaystyle\frac{\partial \phi_0}{\partial s}<0$ for any $s\in(0,s_0]$ (similarly, $\displaystyle\frac{\partial \phi_0}{\partial s}>0$ for any $s\in[-s_0, 0)$). Therefore, $\phi_0$ is injective and it follows that the annulus is necessarily embedded.

\end{proof}

Let us mention that in \cite{CFM}, Cerezo-Fern\'andez-Mira constructed infinitely many free boundary minimal annuli in geodesic balls of $\mathbb S^3$ with a discrete, non-rotational, symmetry group, which are not necessarily embedded. In particular, we can conclude that the family of free boundary minimal annuli in geodesic balls of $\mathbb S^3$ is abundant.

\section{Uniqueness of free boundary minimal annuli in the spherical cap}

In this section we will prove Theorem C stated in the introduction. 

Let $\si^2$ be a compact surface and let $\Phi: (\Sigma,g) \to \mathbb B^{n}(r)$ be a free boundary minimal immersion. We have $\Phi^{*}T\mathbb{S}^{n} = \Phi_{\ast}T\si\oplus N\si$, where $N\si$ is the normal bundle of $\Phi$. Hence, given $V \in \Gamma(\Phi^{\ast}T\mathbb{S}^{n})$, for every $p\in \si$ there exist unique $V^{\top}_{p} \in \Phi_{\ast}T_p\si$ and $V^{\perp}_{p} \in (N\si)_{p}$ such that
$$V = V^{\top} + V^{\perp}.$$

The following connection is induced by $\nabla^{\mathbb{S}^{n}}$:
\begin{eqnarray*}
\nabla &:& \Gamma(\Phi^{\ast}T\mathbb{S}^{n}) \to \Gamma(\Phi^{\ast}T\mathbb{S}^{n}\otimes T^{*}\si),\quad \nabla_{X}V := \nabla^{\mathbb{S}^n}_{\Phi_{\ast}X}V.
\end{eqnarray*}
We say that $V \in \Gamma(\Phi^{\ast}T\mathbb{S}^{n})$ is \emph{conformal} if there exists a function $u\in C^{\infty}(\Sigma)$ such that
\begin{equation*}\label{def:conformal}
\langle\nabla_{X}V,Y\rangle + \langle\nabla_{Y}V,X\rangle = 2u\langle X,Y\rangle,
\end{equation*} 
for any $X,Y \in \Gamma(T\si)$.

For $V,W \in \Gamma(\Phi^{\ast}T\mathbb{S}^{n})$ such that $V_p,W_p \in T_{\Phi(p)}\partial \ba$, $\forall\, p \in \partial\si$, we consider the bilinear form
\begin{equation*}
Q(V,W) = \int_{\si}\big(\langle\nabla V,\nabla W\rangle - 2\langle V,W\rangle + \langle V^{\top},W^{\top}\rangle\big)dA - \cot(r)\int_{\partial\Sigma}\langle V,W\rangle\, dL.
\end{equation*}
Observe that
\begin{equation*}\label{eq:2var.ener}
Q(V,V) = \delta^{2}\mathcal{E}(V),
\end{equation*}
where $\delta^{2}\mathcal{E}$ denotes the second variation of the energy of the harmonic map $\Phi$.

Now assume that $\Sigma$ is an annulus. 
Then, there is $T>0$ such that $(\Sigma,g)$ is conformal to $[-T,T]\times\mathbb S^1$ endowed with the cylindrical metric $g_{\textrm{cyl}}=ds^2+g_{\mathbb{S}^{1}}$, i.e., $g=\lambda g_{\textrm{cyl}}$ for some positive function $\lambda$. Let $\partial_\theta$ be the (globally defined) vector field associated to any angular coordinate. Observe that $\partial_{\theta}$ spans $T_{p}\partial\si$, $\forall\,p \in \partial\si$, and the outward pointing unit $g$-normal of $\partial\si$ is $\pm\lambda^{-1/2}\partial_s$.

\begin{lemma} \label{kernel}
Let $\Phi = (\phi_0,\ldots,\phi_n):(\Sigma,g) \to \mathbb{B}^n(r)$ be a free boundary minimal immersion of an annulus. Then, for any $V \in \Gamma(\Phi^{\ast}T\mathbb{S}^{n})$ such that $V_p \in T_p\pa\ba$, $\forall\,p\in\pa\si$, we have
$$
Q(V,\Phi_{\theta})=0.
$$
\end{lemma}

\begin{proof}
Since $\Delta_g =\lambda^{-1}\Delta_{\textrm{cyl}},$ using Proposition \ref{prop-eigen} we have 
\begin{align*}
& \Delta_{\textrm{cyl}}\,\phi_i+2\lambda \phi_i=0 \quad \Rightarrow \quad \partial_\theta\big(\Delta_{\textrm{cyl}}\,\phi_i+2\lambda \phi_i\big)=0\vspace{0.1cm}\\
\displaystyle\Rightarrow\quad & \displaystyle\Delta_{\textrm{cyl}}\frac{\partial \phi_i}{\partial\theta}+2\frac{\partial\lambda}{\partial\theta}\phi_i+2\lambda\frac{\partial \phi_i}{\partial\theta}=0 \quad
\Rightarrow \quad \left\langle V, \Delta_g\frac{\partial\Phi}{\partial\theta}+2\lambda^{-1}\frac{\partial\lambda}{\partial\theta}\Phi+2\frac{\partial \Phi}{\partial\theta}\displaystyle\right\rangle=0.
\end{align*}\
Since $\langle V, \Phi\rangle =0$ we conclude
\begin{equation}\label{eq:int.init}
\int_{\Sigma}\left\langle V, \Delta_g\frac{\partial\Phi}{\partial\theta}\right\rangle dA+2\int_\Sigma\left\langle V, \frac{\partial \Phi}{\partial\theta}\right\rangle dA=0.
\end{equation}

Write $V = \displaystyle\sum_{j=0}^{n}V_{j}\pa_j$.
Applying the divergence theorem to the vector fields $V_{j}\nabla^{g}\frac{\partial\phi_j}{\partial\theta}$ we obtain 
\begin{equation}\label{eq:lapterm}
\int_{\Sigma}\left\langle V, \Delta_g\frac{\partial\Phi}{\partial\theta}\right\rangle dA = - \sum_{j=0}^{n}\int_{\si}\left\langle \nabla^{g}V_j,\nabla^{g}\frac{\partial\phi_j}{\partial\theta}\right\rangle dA + \sum_{j=0}^{n}\int_{\pa\si}V_{j}\left\langle\nu_{g},\nabla^{g}\frac{\partial\phi_j}{\partial\theta}\right\rangle dL.
\end{equation}

Since $\nu_g=\lambda^{-1/2}\nu_{\textrm{cyl}}$ and $\phi_j$ satisfies \eqref{eq:FBMS3}, for $j=1,\ldots,n$, we get
\begin{align*}\left\langle\nu_{g},\nabla^{g}\frac{\partial\phi_j}{\partial\theta}\right\rangle &= \nu_g\left(\frac{\partial\phi_j}{\partial\theta}\right)=\lambda^{-1/2}\nu_{\textrm{cyl}}\left(\frac{\partial\phi_j}{\partial\theta}\right)\\
&=\lambda^{-1/2}\frac{\partial}{\partial\theta}\left(\nu_{\textrm{cyl}}(\phi_j)\right)=\lambda^{-1/2}\frac{\partial}{\partial\theta}\left(\lambda^{1/2}\nu_{g}(\phi_j)\right)\\
&=\cot(r)\frac{\partial\phi_j}{\partial\theta}+\frac{\partial \lambda^{1/2}}{\partial\theta}\lambda^{-1/2}\cot(r)\phi_j, \quad  \textrm{on } \partial\si.
\end{align*}
Also, notice that $\displaystyle\frac{\partial\phi_0}{\partial\theta}\big\vert_{\partial\si}=0$, since $\phi_0$ is constant along $\partial\si$. So, 
$$
\left\langle\nu_{g},\nabla^{g}\frac{\partial\phi_0}{\partial\theta}\right\rangle= -\frac{\partial \lambda^{1/2}}{\partial\theta}\lambda^{-1/2}\tan(r)\phi_0=- \frac{\partial \lambda^{1/2}}{\partial\theta}\lambda^{-1/2}\sin(r), \quad  \textrm{on } \partial\si.
$$

 Hence,
\begin{align}
\nonumber\sum_{j=0}^{n}\int_{\pa\si}V_{j}\left\langle\nu_{g},\nabla^{g}\frac{\partial\phi_j}{\partial\theta}\right\rangle dL&=\sum_{j=0}^{n}\int_{\pa\si}V_{j}\left[\cot(r)\frac{\partial\phi_j}{\partial\theta}+\frac{\partial \lambda^{1/2}}{\partial\theta}\lambda^{-1/2}\cot(r)\phi_j\right] dL \\
\label{eq:bdrterm} &= \cot(r)\int_{\partial\si}\left\langle V, \frac{\partial\Phi}{\partial\theta}\right\rangle dL,
\end{align}
where in the first equality we used the fact that $V_0=0$ on $\partial \si$, and in the last equality we used that $\langle V, \Phi\rangle =0.$

Consider a local orthonormal frame $\{e_1,e_2\}$ on $\si$. For any $V,W \in \Gamma(\Phi^{\ast}T\mathbb{S}^{n})$ we have
\begin{align*}
\big\langle\nabla V,\nabla W\big\rangle &= \sum_{k} \big\langle \nabla^{\mathbb{S}^n}_{\Phi_{\ast}e_k}V,\nabla^{\mathbb{S}^n}_{\Phi_{\ast}e_k}W\big\rangle\\
&= \sum_{k} \Big\langle \nabla^{\R^{n+1}}_{\Phi_{\ast}e_k}V + \langle\Phi_{\ast}e_k,V\rangle \Phi,\nabla^{\R^{n+1}}_{\Phi_{\ast}e_k}W + \langle\Phi_{\ast}e_k,W\rangle \Phi\Big\rangle \\
&= \sum_{k}\Big[\big\langle \nabla^{\mathbb{R}^{n+1}}_{\Phi_{\ast}e_k}V,\nabla^{\mathbb{R}^{n+1}}_{\Phi_{\ast}e_k}W\big\rangle + \langle\Phi_{\ast}e_k,V\rangle \big\langle\Phi,\nabla^{\mathbb{R}^{n+1}}_{\Phi_{\ast}e_k}W\big\rangle + \langle\Phi_{\ast}e_k,W\rangle \big\langle\Phi,\nabla^{\mathbb{R}^{n+1}}_{\Phi_{\ast}e_k}V\big\rangle\Big]\\
&\quad + \sum_{k}\langle\Phi_{\ast}e_k,V\rangle \langle\Phi_{\ast}e_k,W\rangle\\
&= \sum_{k}\Big[\big\langle \nabla^{\mathbb{R}^{n+1}}_{\Phi_{\ast}e_k}V,\nabla^{\mathbb{R}^{n+1}}_{\Phi_{\ast}e_k}W\big\rangle - 2\langle\Phi_{\ast}e_k,V\rangle \big\langle\Phi_{\ast}e_k,W\big\rangle\Big] + \langle V^{\top},W^{\top}\rangle\\
&= \sum_{j,k} e_k(V_j)e_k(W_j) - \langle V^{\top},W^{\top}\rangle\\
&= \sum_{j} \left\langle \nabla^{g}V_j,\nabla^{g}W_j\right\rangle - \langle V^{\top},W^{\top}\rangle,
\end{align*}
where in the fourth equality we used the fact that $\langle V,\Phi\rangle=\langle W,\Phi\rangle=0.$

Hence,
\begin{equation}\label{eq:derterm}
\sum_{j} \left\langle \nabla^{g}V_j,\nabla^{g}W_j\right\rangle = \big\langle\nabla V,\nabla W\big\rangle + \langle V^{\top},W^{\top}\rangle.
\end{equation}

Using \eqref{eq:lapterm}, \eqref{eq:bdrterm} and \eqref{eq:derterm} in \eqref{eq:int.init}, we get
$$\int_{\si}\bigg[\Big\langle\nabla V,\nabla \frac{\partial\Phi}{\partial\theta}\Big\rangle + \Big\langle V^{\top},\Big(\frac{\partial\Phi}{\partial\theta}\Big)^{\top}\Big\rangle - 2\Big\langle V,\frac{\partial\Phi}{\partial\theta}\Big\rangle\bigg]dA - \cot(r)\int_{\partial\si}\Big\langle V, \frac{\partial\Phi}{\partial\theta}\Big\rangle dL = 0.$$
That is,
$$
Q(V,\Phi_{\theta})=0.
$$
\end{proof}

Now let us specialize to the case $n=3$, and suppose $\si$ is orientable. We know the bundle $N\si$ is trivial, and there is a globally defined unit normal vector $N$ in $\si$.
\begin{lemma}
\label{lemma:perp2}
Let $\Phi: (\Sigma^{2},g) \to \mathbb{B}^{3}(r)$ be a free boundary minimal immersion, where $\si$ is a compact orientable surface. If $V \in \Gamma(\Phi^{\ast}T\mathbb{S}^{n})$ is conformal and $V_p \in T_p\pa\ba$, $\forall\,p\in\pa\si$, then,
\begin{equation}\label{eq:QS.conf}
Q(V,V)= S\big(\langle V,N\rangle,\langle V,N\rangle\big).
\end{equation}
\end{lemma}
\begin{proof}
The first author proved (see the formula at the end of the proof of Theorem 5 in \cite{L}) that
$$(\delta^{2} \mathcal{E})(V) = (\delta^{2} \mathcal{A})(V^{\perp}) + 8\int_{\si}|\zeta|^{2} dA,$$
where $\delta^{2} \mathcal{A}$ is the second variation of area and, with respect to a local isothermal complex coordinate $z = x + iy$, we have
$$\zeta = \Phi_{\ast}\big(\nabla^{g}_{\partial_z} (f\partial_{\bar{z}})\big) + \big(\nabla_{\partial_z} V^{\perp}\big)^{\top},$$
where $V^{\top} = \bar{f}\partial_{z} + f\partial_{\bar{z}}$, and we extended the connections to $\Phi^{\ast}(T\mathbb{S}^3)\otimes_{\R}\mathbb{C}$ complex bilinearly.
However, the condition $\zeta = 0$ is equivalent to the equation (see section 2 of \cite{EM})
$$\langle \nabla_{\partial_z} V,\partial_z\rangle = 0,$$
which can be rewritten as
$$\langle \nabla_{\partial_x} V,\partial_x\rangle = \langle \nabla_{\partial_y} V,\partial_y\rangle \quad \textrm{and}\quad \langle \nabla_{\partial_x} V,\partial_y\rangle = -\langle \nabla_{\partial_x} V,\partial_y\rangle,$$
and this is equivalent to $V$ being conformal. Hence, we have $(\delta^{2} \mathcal{E})(V) = (\delta^{2} \mathcal{A})(V^{\perp})$, and we obtain \eqref{eq:QS.conf}.
\end{proof}

 Observe that if $\si^{2}$ is not contained in a plane passing through the origin, then the space of functions $\mathcal C:=\{\langle \Phi, v \rangle; v\in \mathbb R^{4}\}$ has dimension four.

\begin{lemma}\label{lemma:conformal}
Suppose that $\Phi: (\Sigma,g) \to \mathbb B^3(r)$ is a free boundary minimal immersion of an annulus. Then, there exists a subspace $\mathcal C_1$ of $\mathcal{C}$ of dimension at least three such that for all $\psi\in\mathcal C_1$, there is $Y^{\top} \in \Gamma(T\si)$ such that $Y^{\top}$ is tangential to $\partial \mathbb B^3(r)$ along $\partial \Sigma$ and $Y(\psi)=Y^{\top}+\psi N$ is conformal. Moreover, for any two such vector fields $Y_{1}^{\top}$ and $Y_{2}^{\top}$, there exists a constant $c \in \R$ such that
$$\Phi_{\ast}\big(Y_{1}^{\top} - Y_{2}^{\top}\big) = c\,\frac{\partial\Phi}{\partial\theta}.$$
\end{lemma}
 
\begin{proof}
We need to find $Y^{\top} \in \Gamma(T\si)$ such that $Y(\psi)=Y^{\top}+\psi N$ is conformal, that is, 
\begin{equation*}
\langle\nabla_{\partial_\theta}Y, \partial_s\rangle +\langle\nabla_{\partial_s}Y, \partial_\theta \rangle=0\quad
\textrm{and}\quad
\langle\nabla_{\partial_s}Y, \partial_s\rangle =\langle\nabla_{\partial_\theta}Y, \partial_\theta \rangle.
\end{equation*}

Rewriting these two equations, we get
\begin{align*}
0&=\langle\nabla_{\partial_\theta}(Y^{\top}+\psi N), \partial_s\rangle +\langle\nabla_{\partial_s}(Y^{\top}+\psi N), \partial_\theta \rangle\\
&= \langle\nabla^{g}_{\partial_\theta}Y^{\top}, \partial_s\rangle +\langle\nabla^{g}_{\partial_s}Y^{\top}, \partial_\theta \rangle-2\psi \mathrm{I\!I}_{\si}(\partial_\theta, \partial_s),
\end{align*}
and
\begin{align*}
0&=\langle\nabla_{\partial_s}(Y^{\top}+\psi N), \partial_s\rangle -\langle\nabla_{\partial_\theta}(Y^{\top}+\psi N), \partial_\theta \rangle\\
&=\langle\nabla^{g}_{\partial_s}Y^{\top}, \partial_s\rangle -\langle\nabla^{g}_{\partial_\theta}Y^{\top}, \partial_\theta \rangle -\psi \mathrm{I\!I}_{\si}(\partial_s, \partial_s)+\psi \mathrm{I\!I}_{\si}(\partial_\theta, \partial_\theta)\\
&=\langle\nabla^{g}_{\partial_s}Y^{\top}, \partial_s\rangle -\langle\nabla^{g}_{\partial_\theta}Y^{\top}, \partial_\theta \rangle -2\psi \mathrm{I\!I}_{\si}(\partial_s, \partial_s),
\end{align*}
where in the last equality we used the fact that $\si$ is minimal.

So, we need to solve 
\begin{equation}
\left\{
\begin{array}{rcl}
\langle\nabla^{g}_{\partial_\theta}Y^{\top}, \partial_s\rangle +\langle\nabla^{g}_{\partial_s}Y^{\top}, \partial_\theta \rangle&=& 2\psi \mathrm{I\!I}_{\si}(\partial_\theta, \partial_s), \\
\langle\nabla_{\partial_s}Y^{\top}, \partial_s\rangle -\langle\nabla_{\partial_\theta}Y^{\top}, \partial_\theta \rangle&=&2\psi \mathrm{I\!I}_{\si}(\partial_s, \partial_s).
\end{array}\right.
\label{eq:conf:eq01}
\end{equation}

Recall $\{\partial_s, \partial_\theta\}$ form a (globally defined) basis for $T\si$ and $\vert\partial_s\vert^{2}_{g}=\vert\partial_\theta\vert^{2}_{g}=\lambda$. Hence, we can write $Y^{\top}=u\Phi_s+v\Phi_\theta$ for some functions $u,v$. 

Since $g = \lambda g_{\textrm{cyl}}$, for all $X,Y \in \Gamma(T\si)$,  we have
$$\nabla^{g}_{X}Y = \nabla^{\textrm{cyl}}_{X}Y + \frac{1}{2\lambda}\big[X(\lambda)Y + Y(\lambda)X - g_{\textrm{cyl}}(X,Y)\nabla^{\textrm{cyl}}\lambda\big].$$
Hence, we obtain the following four equations:
\begin{align*}
\langle\nabla^{g}_{\partial_s}Y^{\top}, \partial_s\rangle &=(\partial_s u)\lambda+\frac{1}{2}u(\partial_s\lambda)+\frac{1}{2}v(\partial_\theta \lambda),\\
\langle\nabla^{g}_{\partial_\theta}Y^{\top}, \partial_s\rangle &=(\partial_\theta u)\lambda+\frac{1}{2}u(\partial_\theta\lambda)-\frac{1}{2}v(\partial_s \lambda),\\
\langle\nabla^{g}_{\partial_s}Y^{\top}, \partial_\theta\rangle &=(\partial_s v)\lambda+\frac{1}{2}v(\partial_s\lambda)-\frac{1}{2}u(\partial_\theta \lambda),\\
\langle\nabla^{g}_{\partial_\theta}Y^{\top}, \partial_\theta\rangle &=(\partial_\theta v)\lambda+\frac{1}{2}v(\partial_\theta\lambda)+\frac{1}{2}u(\partial_s \lambda).
\end{align*}
Putting them back in \eqref{eq:conf:eq01}, we obtain
$$
\left\{
\begin{array}{rcl}
\partial_\theta u+\partial_s v&=&2\lambda^{-1}\psi \mathrm{I\!I}_{\si}(\partial_\theta, \partial_s),\\
\partial_s u-\partial_\theta v&=&2\lambda^{-1}\psi \mathrm{I\!I}_{\si}(\partial_s, \partial_s).
\end{array}
\right.
$$
Thus, if we set $f=u+iv$ and $z=s+i\theta,$ these equations become
\begin{equation}
\frac{\partial f}{\partial\bar{z}}=h,
\end{equation}
where $h=\lambda^{-1}\psi \mathrm{I\!I}_{\si}(\partial_\theta, \partial_s)+ i\lambda^{-1}\psi\mathrm{I\!I}_{\si}(\partial_s, \partial_s)$. 
Moreover, the boundary condition that $Y^{\top}$ is tangential to $\partial \mathbb B^3(r)$ along $\partial \Sigma$ means that we need $Y^{\top}$ parallel to $\partial_\theta$ along $\partial\si$, that is, $u=0$ on $\partial\si$. So, we are imposing 
$$\mathfrak{Re}f=0 \mbox{ on } \partial \si.$$

Hence, we need to solve the following boundary linear problem
\begin{equation}
\left\{
\begin{array}{rcl}
\frac{\partial f}{\partial\bar{z}}&=&h, \, \mbox{ in } \si\\
\mathfrak{Re}f&=&0, \, \mbox{ on } \partial \si.
\end{array}
\right.\label{eq:syst}
\end{equation}

This is the same boundary problem as in \cite{F.S3}. By the Fredholm alternative, it has a solution if, and only if, $h$ is $L^{2}(\si,g_{\mathrm{cyl}})$-orthogonal to the kernel of the adjoint of the operator $\frac{\partial}{\partial\bar{z}}$ (defined on the space of complex differentiable functions which are real on the boundary). This kernel consists of the complex differentiable functions satisfying the following:
\begin{equation*}
\left\{
\begin{array}{rcl}
-\frac{\partial \phi}{\partial z}&=&0, \, \mbox{ in } \si\\
\mathfrak{Im}\,\phi&=&0, \, \mbox{ on } \partial \si.
\end{array}
\right.
\end{equation*}
The solutions of this boundary value problem are $\phi \equiv c$, $c \in \R$.

Defining
$$
\mathcal C{_1}=\left\{\psi\in \mathcal{C};\, \mathfrak{Re}\left(\int_\si \big[\lambda^{-1}\psi \mathrm{I\!I}_{\si}(\partial_\theta, \partial_s)+ i\lambda^{-1}\psi\mathrm{I\!I}_{\si}(\partial_s, \partial_s)\big]dA \right)=0 \right\},
$$
it follows from previous analysis that
$\dim \mathcal C{_1} \geq \dim\mathcal C - 1 = 3.$

To prove the uniqueness statement, observe that, if there are two solutions $f_1, f_2$ to \eqref{eq:syst}, then,
$$
\left\{
\begin{array}{rcl}
\frac{\partial{(f_1-f_2)}}{\partial\bar{z}}&=&0,\,  \mbox{ in } \si\\
\mathfrak{Re}(f_1-f_2)&=&0, \, \mbox{ on } \partial \si.
\end{array}
\right.
$$
So, the maximum principle implies that $f_1-f_2=ic$ in $\si$ for some real constant $c,$ and the statement follows.
\end{proof}

Now we are able to prove the uniqueness result stated in the introduction.

\begin{thm3}
\label{thm:uniq}
Let  $\Sigma$ be an annulus and consider $\Phi = (\phi_0,\ldots,\phi_n):(\Sigma,g) \to \mathbb{B}^n(r)$ a free boundary minimal immersion. Suppose $\phi_j$ is a $\sigma_{1}(g)$-eigenfunction, for $j=1,\ldots,n$. Then $n=3$ and $\Phi(\Sigma)$ is one of the rotational examples given by \eqref{eq-annulus}.
\end{thm3}

\begin{proof}
Recall that the coordinate functions $\phi_i = \langle \Phi,\partial_i\rangle$ satisfy the equations \eqref{eq:FBMS1}, \eqref{eq:FBMS2}, \eqref{eq:FBMS3},
and we are assuming that $\sigma_0=-\tan(r)$ and $\sigma_1=\cot(r).$ The multiplicity bound on $\sigma_1$  (see Theorem \ref{thm:multiplicity} and Remark \ref{rem-mult}) implies that $n=3$ necessarily. We divide the proof into two cases.\\

\noindent
{\bf Case 1:} Suppose $\displaystyle\int_{\partial\si}\frac{\partial\Phi}{\partial\theta} dA=0$.\\

First observe that $\displaystyle\frac{\partial\Phi}{\partial\theta}\in \Gamma(\Phi*T\mathbb S^n)$ and, since the annulus is free boundary, we know $\displaystyle\frac{\partial\Phi}{\partial\theta}\in T\pa\ba$ on $\pa\si$.

The calculations in the proof of Lemma \ref{kernel} imply that
 \begin{equation}\label{eq:25}
 \Delta_g\frac{\partial\Phi}{\partial\theta}+2\lambda^{-1}\frac{\partial\lambda}{\partial\theta}\Phi+2\frac{\partial \Phi}{\partial\theta}=0 \quad \textrm{in } \si
 \end{equation}
 and (taking $V=\displaystyle\frac{\partial\Phi}{\partial\theta}$)
 \begin{equation}\label{eq:22}
 \sum_{j=0}^{n}\int_{\si} \left|\nabla^{g}\frac{\partial\phi_j}{\partial\theta}\right|^2 dA + \sum_{j=0}^{n}\int_{\partial\si}\frac{\partial\phi_j}{\partial\theta}\, \nu_g\left(\frac{\partial\phi_{j}}{\partial\theta}\right) dL-2\int_\Sigma\left| \frac{\partial \Phi}{\partial\theta}\right|^2 dA=0,
 \end{equation}
where 
 $$\nu_g\left(\frac{\partial\phi_i}{\partial\theta}\right)=\sigma_1\frac{\partial\phi_i}{\partial\theta}+\frac{\partial \lambda^{1/2}}{\partial\theta}\lambda^{-1/2}\sigma_1\phi_i, \quad  \textrm{on } \partial\si \quad \textrm{for } i=1,2,3.$$ 
Since $\displaystyle\frac{\partial\phi_0}{\partial\theta}\Big\vert_{\partial\si}=0$ and $\displaystyle\left\langle \Phi, \frac{\partial\Phi}{\partial\theta} \right\rangle = 0$, we can rewrite \eqref{eq:22} as

\begin{equation}\label{eq:24}
\begin{array}{rl}
&\displaystyle\sum_{i=1}^3\left[\int_\si  \left(   \left|\nabla^{g}\frac{\partial\phi_i}{\partial\theta}\right|^2  -2\left| \frac{\partial \phi_i}{\partial\theta}\right|^2 \right)dA  -\sigma_1\int_{\partial\si}  \left| \frac{\partial \phi_i}{\partial\theta}\right|^2dL\right] + \vspace{0.1cm}\\
&\displaystyle \int_\si\left(\left|\nabla^{g}\frac{\partial\phi_0}{\partial\theta}\right|^2  -2\left| \frac{\partial \phi_0}{\partial\theta}\right|^2 \right)dA  -\sigma_0\int_{\partial\si}  \left| \frac{\partial \phi_0}{\partial\theta}\right|^2dL=0.
\end{array}
\end{equation}

First observe that since $\displaystyle\frac{\partial \phi_0}{\partial\theta}\Big\vert_{\partial\si} = 0,$ Proposition \ref{prop-dirichlet} implies that the last term in the above expression is nonnegative.

Moreover, since $\phi_0\big\vert_{\partial\si}$ is constant and we are assuming $\phi_0$ is a first eigenfunction, the condition $\displaystyle\int_{\partial\si}\frac{\partial\Phi}{\partial\theta} dA=0$ implies that, if $\displaystyle\frac{\partial \phi_i}{\partial\theta}\big\vert_{\partial\si} \neq 0$, then $\displaystyle\frac{\partial\phi_i}{\partial\theta}$ is a test function for the variational characterization of $\sigma_1,$ in particular, we have
$$
\int_\si  \left(\left|\nabla^{g}\frac{\partial\phi_i}{\partial\theta}\right|^2  -2\left| \frac{\partial \phi_i}{\partial\theta}\right|^2 \right)dA  -\sigma_1\int_{\partial\si}  \left| \frac{\partial \phi_i}{\partial\theta}\right|^2dL\geq0.
$$
For the case where $\displaystyle\frac{\partial \phi_i}{\partial\theta}\big\vert_{\partial\si} = 0$, we get the same conclusion above using Proposition \ref{prop-dirichlet}.

Therefore, \eqref{eq:24} implies that all inequalities above are equalities. Hence, if $\displaystyle\frac{\partial \phi_i}{\partial\theta}\Big\vert_{\partial\si} \neq 0,$ then $\displaystyle\frac{\partial \phi_i}{\partial\theta}$ is a (nontrivial)  $\sigma_1$-eigenfunction. In any case, it holds
$$
\Delta_g\frac{\partial\phi_i}{\partial\theta}+2\frac{\partial \phi_i}{\partial\theta}=0 \textrm{ in } \si \textrm{ for } i=1,2,3.
$$
Hence, by \eqref{eq:25} we get
$$
\lambda^{-1}\frac{\partial\lambda}{\partial\theta}\phi_i=0  \textrm{ in } \si  \textrm{ for } i=1,2,3.
$$
Since $\displaystyle\sum_{i=1}^3\phi_i^2\neq0$ in $\si\setminus\{(1,0,0,0)\}$ we conclude that $\displaystyle\frac{\partial\lambda}{\partial\theta}\equiv0$ in $\si$, that is, the metric is rotational.\\

\noindent
{\bf Case 2:} Suppose $\displaystyle\int_{\partial\si}\frac{\partial\Phi}{\partial\theta} dA\neq0$.\\

By Lemma \ref{lemma:conformal}, for any $\psi\in \mathcal C_1$, there is a conformal vector field whose normal component is $\psi N$, and which is unique up to addition of a real multiple of $\Phi_{\theta}$. We denote by $Y(\psi) = Y^\top +\psi N$ the unique conformal vector field such that 
$\displaystyle\left\langle \int_{\partial\si}Y(\psi)dL, \int_{\partial\si}\Phi_{\theta} dL\right\rangle=0$.

Also, by Lemma  \ref{lemma:conformal}, we have that $Y^\top$ is tangential to $\partial \mathbb B^3(r)$. Hence, in particular, since the unit normal vector to $\partial \mathbb B^3(r)$ at a point $x$ is given by $(\sin r)^{-1}\big(\langle \partial_0, x\rangle x-\partial_0\big)$, we conclude that $\langle Y(\psi), \partial_0\rangle=0.$ 

By construction, the map $\psi\to Y(\psi)$ is linear. Consider the vector space 
$$
\mathcal{V}=\left\{Y(\psi)+c\,\Phi_{\theta}; \psi\in \mathcal{C}_1, c\in\mathbb{R}\right\}.
$$
Notice that $\mathcal{V}$ has dimension at least $4$, since dim$(\mathcal{C}_1)\geq 3$ and any nonzero vector field $Y(\psi)$ has a nontrivial normal component, while $\Phi_{\theta}$ is tangential to $\si.$ Also, every element in $\mathcal{V}$ has first coordinate zero, i.e., $\left\langle Y(\psi)+c\,\Phi_{\theta}, \partial_0\right\rangle=0,$ for any $\psi\in \mathcal{C}_1, c\in\mathbb{R}.$

Define a linear transformation 
$T: \mathcal{V} \to \{0\}\times\mathbb{R}^3\subset \mathbb{R}^4$ by
$$T\left(Y(\psi)+c\,\Phi_{\theta}\right)=\int_{\partial\si}\left(Y(\psi)+c\,\Phi_{\theta}\right) dL.$$

Since dim$(V)\geq 4$, this linear transformation has a nontrivial nullspace, i.e., there exist $\psi\in \mathcal{C}_1$ and $c\in\mathbb{R}$ such that $\int_{\partial\si}\left(Y(\psi)+c\,\Phi_{\theta}\right) dL=0$ and $Y(\psi)+c\,\Phi_{\theta}\neq0.$

We have
$$
Q\left(Y(\psi)+c\,\Phi_{\theta}, Y(\psi)+c\,\Phi_{\theta}\right)=Q\big(Y(\psi), Y(\psi)\big)+2c\,Q\left(Y(\psi), \Phi_{\theta}\right)+c^2Q\left(\Phi_{\theta}, \Phi_{\theta}\right).
$$

By Lemma \ref{kernel}, we know 
$$Q\left(Y(\psi), \Phi_{\theta}\right)=Q\left(\Phi_{\theta}, \Phi_{\theta}\right)=0,$$
and, since $Y(\psi)$ is conformal, Lemmas \ref{lemma:perp} and \ref{lemma:perp2} give that
$$
Q(Y(\psi),Y(\psi))= S(\psi,\psi)= -\int_\si \vert\mathrm{I\!I}_{\si}\vert^2\psi^2 dA-\langle v,\partial_0\rangle^2\cot(r)L(\partial\si).
$$
Hence,
\begin{equation}\label{eqQsign}
Q\left(Y(\psi)+c\,\Phi_{\theta}, Y(\psi)+c\,\Phi_{\theta}\right)=-\int_\si \vert\mathrm{I\!I}_{\si}\vert^2\psi^2dA -\langle v,\partial_0\rangle^2\cot(r)L(\partial\si)\leq 0.
\end{equation}

Denote
$$f_i=\left(Y(\psi)+c\,\Phi_{\theta}\right)_i=\left\langle Y(\psi)+c\,\Phi_{\theta}, \partial_i\right\rangle,\ i=0,1,2,3.$$ 
Since $f_0\big\vert_{\partial\si}=0,$ we know by Proposition \ref{prop-dirichlet} that
$$
\int_{\si}\left(\vert \nabla^{g} f_0\vert^2-2f_0^2\right)dA\geq 0.
$$
On the other hand, since $\displaystyle\int_{\partial\si}\left(Y(\psi)+c\,\Phi_{\theta}\right) dL=0$, for $i=1,2,3,$ either the function $f_i$
is a test function for the variational characterization of $\sigma_1,$ which implies that
$$
\int_{\si}\left(\vert \nabla^{g} f_i\vert^2-2f_i^2\right)dA-\cot(r)\int_{\partial\si}f_i^2 dL\geq0;
$$
or
$f_i\vert_{\pa\si} = 0$ and, by Proposition \ref{prop-dirichlet}, we have
$$
\int_{\si}\left(\vert \nabla^{g} f_i\vert^2-2f_i^2\right)dA\geq 0.
$$

Then, 
\begin{equation}\label{eqQsign2}
Q\left(Y(\psi)+c\,\Phi_{\theta}, Y(\psi)+c\,\Phi_{\theta}\right)=\sum_{i=0}^3\left(\int_{\si}\left(\vert \nabla^{g} f_i\vert^2-2f_i^2\right)dA-\cot(r)\int_{\partial\si}f_i^2 dL\right)\geq 0. 
\end{equation}

Thus, \eqref{eqQsign} and \eqref{eqQsign2} yield that
$$
-\int_\si \vert\mathrm{I\!I}_{\si}\vert^2\psi^2dA= 0,
$$
which implies $\vert\mathrm{I\!I}_{\si}\vert^2\psi^2 =0$. If $\psi > 0$ at some point $p$, then the same holds in a neighborhood of $p$, which implies $\vert\mathrm{I\!I}_{\si}\vert \equiv 0$ in $\si$ (by analytic continuation). However, this contradicts the fact that $\si$ is an annulus. Thus, $\psi \equiv 0$ and, hence, $Y(\psi)=0$. Therefore, $\displaystyle\int_{\partial\si}\Phi_{\theta} dA=0$ and we are in the situation of our previous case, where we can conclude that the metric has the form $\lambda(s)g_{\mathrm{cyl}}$.

Now we are going to show that this implies the surface $\si$ is rotational.

In the proof of Cases 1 and 2 above, we concluded that for $i=1,2,3$, either $\displaystyle\frac{\partial\phi_i}{\partial\theta}$ is identically zero or is a $\sigma_1$-eigenfunction. Since the multiplicity of $\sigma_1$ is three and $\{\phi_1,\phi_2, \phi_3\}$ form a basis for its eigenspace, there are constants $a_{ij}$ such that
$
\displaystyle\frac{\partial\phi_i}{\partial\theta}=\sum_{j=1}^3 a_{ij}\phi_j.
$
Also, observe that  $\displaystyle\frac{\partial\phi_0}{\partial\theta}\Big\vert_{\pa\si}=0$ and, since $\displaystyle\frac{\pa\lambda}{\pa\theta}=0$, $\displaystyle\Delta_g \frac{\partial\phi_0}{\partial\theta}+2\frac{\partial\phi_0}{\partial\theta}=0$ in $\si$. Then, since $2$ is not a Dirichlet eigenvalue, we conclude that $\displaystyle\frac{\partial\phi_0}{\partial\theta}=0$ on $\si.$

Let us denote by $A$ the $4\times 4$ matrix with constant entries defined by 
\begin{equation}\label{eq:def.A}
\frac{\partial\Phi}{\partial\theta}(s,\theta)= A\Phi(s,\theta).
\end{equation}
Our goal is to show that $A$ is antisymmetric.

Observe that $\langle A\Phi, \Phi\rangle=\langle \Phi_{\theta}, \Phi\rangle=0.$ Hence, differentiating the left-hand side of this equality, we conclude that 
$$
\langle A\Phi (s,\theta), v\rangle +\langle \Phi (s,\theta),Av\rangle=0,
$$
for any $v\in T_{\Phi(s,\theta)}\Phi(\si)$.

Now define a symmetric constant matrix $B$ by
$$
\langle Bv,w\rangle = \langle Av,w\rangle +\langle v,Aw\rangle  \mbox{ for any } v, w\in \mathbb R^4.
$$
Note that $\langle B\Phi,\Phi\rangle=0$ and $\langle B\Phi,Y\rangle=0$ for any $Y\in T_{\Phi(s,\theta)}\Phi(\si).$ Thus, there is a function $a=a(s,\theta)$ such that 
$
B\Phi=aN,
$
where $N=N(s,\theta)$ is the unit normal to $\si$ at $\Phi(s,\theta).$

Since $\Phi_{\theta}= A\Phi,$ we have $A\Phi_{s}=\Phi_{s\theta}.$ In particular,
$$
\left\langle A\Phi_{s}, \Phi_{s}  \right\rangle = \left\langle \Phi_{s\theta}, \Phi_{s}  \right\rangle=\frac{1}{2}\frac{\partial}{\partial\theta}\vert \Phi_{s}\vert^2=\frac12\frac{\partial}{\partial\theta}(\lambda(s))=0.
$$
Thus,
$$
\left\langle B\Phi_{s}, \Phi_{s} \right\rangle= 2\big\langle A\Phi_{s}, \Phi_{s}  \big\rangle=0.
$$

Using that $B=aN,$ we have
$$
B\Phi_{s}=\frac{\partial a}{\partial s}N +a \nabla_{\Phi_s}^{\mathbb S^3}N,
$$
so,
$$
0=\left\langle B\Phi_{s}, \Phi_{s} \right\rangle=a \big\langle  \nabla_{\Phi_s}^{\mathbb S^3}N, \Phi_{s} \big\rangle.
$$
Similarly, using that $ \displaystyle A\Phi_{\theta}=\Phi_{\theta\theta}$, we get
$$
0=\left\langle B\Phi_{\theta}, \Phi_{\theta} \right\rangle=a \big\langle  \nabla_{\Phi_\theta}^{\mathbb S^3}N, \Phi_{\theta} \big\rangle.
$$

Using the fact that $\Delta_{\textrm{cyl}} \Phi+2\lambda\Phi=0,$ we have
\begin{align*}
a \left\langle  \nabla_{\Phi_s}^{\mathbb S^3}N, \Phi_{\theta} \right\rangle &= \displaystyle\left\langle B\Phi_s, \Phi_{\theta} \right\rangle =  \displaystyle\left\langle A\Phi_s, \Phi_{\theta} \right\rangle+ \left\langle \Phi_s, A\Phi_{\theta} \right\rangle\\
&=  \displaystyle\left\langle \Phi_{s\theta}, \Phi_{\theta} \right\rangle + \left\langle \Phi_{\theta\theta}, \Phi_s \right\rangle= \displaystyle\left\langle \Phi_{s\theta}, \Phi_{\theta} \right\rangle +\left\langle -\Phi_{ss} -2\lambda\Phi, \Phi_s \right\rangle\\
&= \displaystyle\frac{1}{2}\frac{\partial}{\partial s} \vert \Phi_{\theta}\vert^2-\frac{1}{2}\frac{\partial}{\partial s} \vert \Phi_s\vert^2= 0,
\end{align*}
where the last equality follows from conformality.

If $a(s,\theta)\neq0$ for some $(s,\theta),$ then there would be an open neighborhood of $(s,\theta)$ such that $a(p)\neq 0$ for all $p\in U.$ Hence, by the equalities obtained above, we would have 
$$
\big\langle  \nabla_{\Phi_s}^{\mathbb S^3}N, \Phi_{\theta} \big\rangle = \big\langle  \nabla_{\Phi_s}^{\mathbb S^3}N, \Phi_{s} \big\rangle = \big\langle  \nabla_{\Phi_\theta}^{\mathbb S^3}N, \Phi_{\theta} \big\rangle =0\ \mbox{ in } U,
$$
which would imply that $\si$ is totally geodesic, a contradiction, since $\si$ is an annulus. Therefore, $a\equiv 0$ and we conclude that $Bx=0$ for all $x\in \si.$

Suppose $B$ is not the zero matrix. Then $\si$ is contained in the hyperplane Ker$(B)$ which implies that $\si$ is totally geodesic, a contradiction. Therefore, $B$ is the zero matrix and $A$ is antisymmetric. 

Observe that by the definition of $A$, we know that $Ae_0=0;$ hence, $0$ is an eigenvalue of $A$ and, since it is a $4\times 4$ antisymmetric matrix, it has multiplicity at least $2$. Then, there exists a unit vector $v$ orthogonal to $e_0$ such that $Av=0.$  Now consider the space $W=(\mbox{span}\{e_0, v\})^\perp$. Since $A$ is antisymmetric, this space is invariant by A, i.e., $A(W)\subset W.$ 

Hence, if we take $\{w_1,w_2\}$ an orthonormal basis of $W$, then $\{e_0,v,w_1,w_2\}$ is an orthonormal basis for $\mathbb R^4$ and the matrix $A$ is written with respect to this basis as:
$$
A=\begin{pmatrix}
0 & 0 & 0 & 0\\
0 & 0 & 0 & 0\\
0 & 0 & 0 & -\mu\\
0 & 0 & \mu & 0
\end{pmatrix} \mbox{ for some } \mu\neq0.
$$

Now, for any $t$, consider 
$$
R_t=\begin{pmatrix}
1 & 0 & 0 & 0\\
0 & 1 & 0 & 0\\
0 & 0 & \cos t & -\sin t\\
0 & 0 & \sin t & \cos t
\end{pmatrix}.
$$

Observe that $\frac{d}{dt}(R_t)=\frac{A}{\mu}(R_t),$ i.e., the one-parameter subgroup generated by $A$ is given by rotations. Also, equation \eqref{eq:def.A} implies $Ax\in T_x\si$ for any $x\in\si$ and $Ay\in T_y\partial\si$ for any $y\in\partial\si$, thus the matrix $\mu^{-1}A$ defines two vector fields $X=\mu^{-1}A: \si \to T\si$ and $Y=\mu^{-1}A\Big\vert_{\pa\si}: \pa\si \to T(\pa\si)$. 
We claim that $R_t(\partial\Sigma)\subset \partial\Sigma$ and $R_t(\Sigma)\subset\Sigma$, i.e, $\si$ is rotational. In fact, given $p\in \partial\si,$ there exists a unique curve $\beta(t)\subset \partial \si$ such that $\beta(0)=p$ and $\beta'(t)=Y(\beta(t))=X(\beta(t))$. On the other hand, $\tilde{\beta}(t) =  R_{t}(p)$ is a curve such that $\tilde{\beta}(0)=p$ and  $\tilde{\beta}'(t)=X(\tilde{\beta}(t))$. Then, by uniqueness of solutions to ODE's, we have $\tilde{\beta}(t)=\beta(t).$ Therefore, $R_t(\partial\Sigma)\subset \partial\Sigma$. In a similar way we can show $R_t(\Sigma)\subset\Sigma$. 

By the result of do Carmo and Dajczer \cite{dCD} (see also \cite{Ot}), $\si$ is one of the rotational examples given by \eqref{eq-annulus}.
\end{proof}

\end{document}